\documentclass{amsart}
\usepackage{amsrefs}
\usepackage{bbold,math,tikz,mathrsfs,lineno,pgf,xypic}

\newcommand\FF{{\mathbb F}}
\newcommand\NN{{\mathbb N}}
\newcommand\ZZ{{\mathbb Z}}
\newcommand\QQ{{\mathbb Q}}
\newcommand\KK{{\mathbb K}}
\newcommand\Qb{{\overline\QQ}}
\newcommand\CC{{\mathbb C}}
\newcommand\Rat{\operatorname{\textsf{Rat}}}
\newcommand\PSL{\operatorname{\textsf{PSL}}}
\newcommand\Pone{{\mathbb P^1}}

\theoremstyle{definition}
\newtheorem{plainnotation}[thm]{Notation}
\newenvironment{notation}{\pushQED{\defqed}\begin{plainnotation}}{\popQED\end{plainnotation}}
\newtheorem{alg}[thm]{Algorithm}
\newtheorem{example}[thm]{Example}

\def\sphere{{\mathbb S}}

\begin{document}

\author[L. Bartholdi]{Laurent Bartholdi}
\address{L.B.: Mathematisches Institut, Georg-August Universit\"at zu G\"ottingen}

\author[X. Buff]{Xavier Buff}
\address{X.B.: Institut de Math\'ematiques de Toulouse, Universit\'e Paul Sabatier, Toulouse}

\author[C. Bothmer]{Hans-Christian Graf von Bothmer}
\author[J. Kr\"oker]{Jakob Kr\"oker}
\address{H-.C.G.v.B., J.K.: Courant Research Centre ``Higher Order
  Structures'', Georg-August Universit\"at zu G\"ottingen}

\thanks{The authors acknowledge support from the Courant Research
  Centre ``Higher Order Structures'' at Georg-August Universit\"at zu
  G\"ottingen and the CIMI at Toulouse}

\date{25 October 2013}
\title{Algorithmic Construction of Hurwitz Maps}
\begin{abstract}
  We describe an algorithm that, given a $k$-tuple of permutations
  representing the monodromy of a rational map, constructs an
  arbitrarily precise floating-point complex approximation of that
  map.

  We then explain how it has been used to study a problem in dynamical
  systems raised by Cui.
\end{abstract}
\maketitle

\section{Introduction}
Let $\sphere$ be a topological oriented $2$-sphere.  The branched
coverings $\sphere\to \sphere$ considered in this article are all
orientation preserving.  Let $Q:=\{Q_i\}_{i\in I}$ be a finite subset
of $\sphere$ with $I=\ZZ/k\ZZ$. In~\cite{hurwitz:ramifiedsurfaces},
Hurwitz describes an elegant classification of branched coverings
$\sphere\to \sphere$ with critical values contained in $Q$ in terms
of \emph{admissible} $k$-tuples of permutations $(\sigma_i\in \sym
d)_{i\in I}$. A $k$-tuple is admissible if:
\begin{itemize}
\item the permutations $(\sigma_i)_{i\in I}$ generate a transitive
  subgroup of $\sym d$, 
\item $\sigma_{1}\cdot \sigma_{2}\cdots\sigma_{k}={\rm id}$ and 
\item the cycle lengths satisfy the condition
  \begin{equation}\label{eq:riemannhurwitz}
    \sum_{i\in I} \sum_{\substack{c\text{ cycle}\\\text{of }\sigma_i}}\bigl(\text{length}(c)-1\bigr)=2d-2.
  \end{equation}
\end{itemize}
See~\S\ref{ss:hurwitz} for more details regarding the classification. 

It is easy, using a computer algebra system such as
\textsc{Gap}~\cite{gap4:manual}, to enumerate all admissible
$k$-tuples of permutations; it is an altogether different problem to
\emph{construct} an analytic model of a covering associated to a given
admissible $k$-tuple of permutations.  The purpose of this note is to
describe such an algorithm and its implementation.

\subsection{Hurwitz's classification}\label{ss:hurwitz}
Two branched coverings $f:\sphere\to \sphere$ and $g:\sphere\to
\sphere$ are equivalent if there is an orientation preserving
homeomorphism $h:\sphere\to \sphere$ such that $g =
f\circ h$. Hurwitz's result is a classification of equivalence classes of
coverings in this sense.

Choose a basepoint $*\in \sphere \setminus Q$. For
each $i\in I$, choose a path $\gamma_i$ joining $*$ to $Q_i$ in $\sphere\setminus Q$, in such a way that 
\begin{itemize}
\item the paths $\gamma_i$ intersect only at $*$, 
\item the paths $(\gamma_{1},\ldots,\gamma_{k})$ are ordered cyclically counterclockwise around $*$. 
\end{itemize}
The fundamental group $G=\pi_1(\sphere\setminus Q,*)$ is generated by
paths $\hat\gamma_i$ that follow $\gamma_i$, wind once counterclockwise
around $Q_i$, and return to $*$ along $\gamma_i$. It has the presentation
\[G=\langle \hat\gamma_i, i\in I\mid \hat\gamma_{1}\cdot\hat\gamma_{2}\cdots \hat\gamma_{k}={\rm id}\rangle.\]

Let $f:(\sphere,C)\to (\sphere,Q)$ be a covering branched over $Q$. Number
$\{*_1,\dots,*_d\}$ the $f$-preimages of $*$. Then, for each $i\in I$ and each
$m\in\{1,\dots,d\}$, the path $\hat \gamma_i$ lifts to a path starting at $*_m$ and ending at $*_n$ for some $n=:\sigma_i(m)$. This defines a permutation $\sigma_i$ for each $i\in I$. 
Note that the $k$-tuple $(\sigma_i)_{i\in I}$ is admissible: 
\begin{itemize}
\item since $\sphere\setminus C$ is connected, the group
  $\langle\sigma_i\rangle$ is transitive on $\{1,\dots,d\}$;
\item since $\hat\gamma_{1}\cdot\hat\gamma_{2}\cdots
  \hat\gamma_{k}={\rm id}$, we have that $\sigma_{1}\cdot
  \sigma_{2}\cdots\sigma_{k}={\rm id}$;
\item computing the Euler characteristic of $\sphere\setminus C$ via
  the Riemann-Hurwitz formula yields~\eqref{eq:riemannhurwitz}.
\end{itemize}

Conversely, let $(\sigma_i)_{i\in I}$ be an admissible $k$-tuple of
permutations.  Define a branched covering as follows: start with $d$
disjoint copies of $\sphere$, cut open along the paths $\gamma_i$. If
$\sigma_i(m)=n$, glue the right boundary of $\gamma_i$ on $m$-th
sphere to the left boundary of $\gamma_i$ on the $n$-th sphere. This
defines a covering with critical values contained in $Q$. It is
connected because $\langle\sigma_i\rangle$ is transitive on
$\{1,\dots,d\}$. The Euler characteristic of the cover is $2$, because
of~\eqref{eq:riemannhurwitz} and the Riemann-Hurwitz formula; so it is
a sphere.

The $k$-tuple $(\sigma_i)_{i\in I}$ must be considered up to diagonal
conjugation by $\sym d$, which amounts to numbering the spheres
differently. The constructions above then define a bijection between
equivalence classes of branched coverings and equivalence classes of
appropriate $k$-tuples of permutations.

A coarser equivalence relation on coverings has also been considered,
but is not the main focus of this article: two coverings
$f,g:\sphere\to\sphere$ are \emph{Hurwitz equivalent} if there exist
homeomorphisms $h_0,h_1:\sphere\to\sphere$ with $f\circ h_1=h_0\circ
g$. Hurwitz classes of coverings may also be classified by $k$-tuples
of permutations; namely, by the orbits on appropriate $k$-tuples of
the symmetric group $\sym d$ (acting as above) and the pure braid
group on $k$ strings. The latter group's generators act by
conjugating, for any two consecutive points $Q_i,Q_{i+1}$ in $Q$, the
permutations $\sigma_i$ and $\sigma_{i+1}$ by
$\sigma_i\sigma_{i+1}$. This amounts to changing the ``spider''
$\bigcup_{i\in I}\gamma_i$ by twisting the legs $\gamma_i$ and
$\gamma_{i+1}$ around each other.

\subsection{Analytic models}
Assume now $Q\subset \Pone(\CC)$ and that $f:\sphere\setminus C\to
\Pone(\CC)\setminus Q$ is a covering map. Then, $f$ defines
holomorphic charts on $\sphere\setminus C$ and it is not difficult to
see that the points in $C$ are removable singularities: we denote by
$\sphere_f$ the corresponding Riemann surface.  By the Uniformization
Theorem, there is a conformal homeomorphism $\phi_f:\sphere_f\to
\Pone(\CC)$.  The map $F:=f\circ \phi_f^{-1}:\Pone(\CC)\to \Pone(\CC)$
is a holomorphic branched covering, i.e., a rational map.  Assume
$g=f\circ h:\sphere\to \Pone(\CC)$ for some homeomorphism
$h:\sphere\to \sphere$. Let $\phi_g:\sphere_g\to \Pone(\CC)$ be a
conformal homeomorphism and set $G:=g\circ \phi_g^{-1}:\Pone(\CC)\to
\Pone(\CC)$ be the corresponding rational map. Then, $H=\phi_f\circ
\phi_g^{-1}:\Pone(\CC)\to \Pone(\CC)$ is a M\"obius transformation and
$F=G\circ H$.

Therefore, up to precomposition by a M\"obius transformation, the
rational map $F$ only depends on the equivalence class of covering
$f:(\sphere, C)\to (\Pone(\CC),Q)$. We say that $f$ is an
\emph{analytic model}.

\subsection{Dynamics}\label{ss:dynamics}
Our algorithm is an important step in the more difficult problem of
determining an analytic model with given dynamics. We start by
recalling some definitions.  The \emph{post-critical set} of a
branched self-covering $f:\sphere\to \sphere$ with critical value set
$Q_f$ is
\[P_f:=\bigcup_{n\ge0} f^{\circ n}(Q_f).
\]
We are interested in the case where $P_f$ is finite and we consider
$f$ up to isotopy rel $P_f$; namely, we that that $f$ and $g$ are
\emph{combinatorially equivalent}, and write $f\sim g$, if there
exists a path of branched self-coverings from $f$ to $g$ whose
post-critical set moves smoothly.

The dynamical problem alluded to above asks to determine, given a
branched covering $f:\sphere\to\sphere$ with finite post-critical set,
whether there exists a rational map that is combinatorially equivalent
to $f$, and in that case to exhibit such a rational map.

A fundamental theorem of Thurston (see~\cite{douady-h:thurston} and
Theorem~\ref{thm:thurston} below) proves (except in few
well-understood, low-complexity cases) that such a branched covering
$f$ is combinatorially equivalent to at most one rational map, up to
conjugation by a M\"obius transformation; furthermore, if $\#Q=3$,
then it has precisely one holomorphic realization.

In case $f$ is a \emph{topological polynomial} (it has a fixed point
of maximal ramification), the dynamics of $f$ may be described by
combinatorial data called ``external rays'',
see~\cite{poirier:portraits}. An implementation, when $f$ has only two
critical values, is described in~\cite{hubbard-s:spider}, and is
called the ``spider algorithm''; see also~\cite{boyd-henriksen:medusa}
treating the general degree-$2$ case. In a forthcoming paper, the
first author will describe the implementation of the general case.

If $\#Q=3$, then we may assume $Q_1=\infty$, $Q_2=0$ and $Q_3=1$
within $\Pone(\CC)$. Furthermore, precomposing $f$ by an appropriate
M\"obius transformation, we may also assume that $Q\subset
f^{-1}(Q)$. In the polynomial case, Pilgrim linked
in~\cite{pilgrim:dessins} the ``dessin d'enfant'' (the full preimage
of the segment $[0,1]$) of $f$ with a dynamical invariant, its
``Hubbard tree''.

We describe in Section~\ref{sec:cui} a question by Cui in the theory
of holomorphic dynamical systems, and give an explicit holomorphic
realization of a topological map he constructed.

This will also be our running example in the text. With
$Q_1=\infty$, $Q_2=0$ and $Q_3=1$, the permutations representing the map are
\begin{alignat}{2}
  &\sigma_1&=(1,7,11,2)(3,8)\underline{(4,5)}(6,10)(9,12,13),\notag\\
  &\sigma_2 &=(1,3,12,4)(5,9)\underline{(6,7)}(10,13,11)(2,8),\label{eq:perm}\\
  &\sigma_3 &=(1,5,13,6)(7,10)\underline{(2,3)}(8,11,12)(4,9).\notag
\end{alignat}
Recall that the cycles of the above permutations correspond to
preimages of critical values. We seek a degree-$13$ rational map $f$
such that the \underline{underlined} cycle $(4,5)$ and its image under
$f$ are located at $\infty$, and similarly for the other two cycles
and images.

In this specific example, the search can be made more feasible as
follows. Setting all critical points as unknowns and eliminating is
out of the question. With a little faith that the symmetry between
$\infty,0,1$ translates to $f$, let $\rho(z)=1/(1-z)$ be the rotation
permuting $\infty,0,1$, and note that $\Pone(\CC)/\langle\rho\rangle$
is a sphere, branched at the two fixed points of $\rho$. If $f$
descends to a map $g$ on $\Pone(\CC)/\langle\rho\rangle$, then (after
change of variables) it has the form $g(z)=z(p(z)/q(z))^3$ for
degree-$4$ polynomials, such that $g(z)=1+\mathcal O((z-1)^4)$ at
$z=1$, and such that $1$ is the image of four other points with local
degrees $3,2,2,2$ respectively. We are grateful to Noam Elkies and
Curt McMullen for having pointed out to us the feasibility of this
approach.

Nevertheless, we will show that our algorithm is strong enough to
produce a solution even without exploiting the symmetry of the Hurwitz
data.

\subsection{Simple cases}
If $\#Q=2$, then there is a unique solution represented, up to diagonal conjugation, by the pair of permutations
\[\sigma_1=(1,2,\ldots,d)\quad\text{and}\quad \sigma_2=(d,\ldots,2,1).\]
If $Q_1=\infty$ and $Q_2=0$, an analytic model is $f(z)=z^d$.  

However, the case $\#Q=3$ seems already as complicated as
the general case, and has only been addressed in the literature for
small $d$. Such maps are often called ``dessins d'enfant'',
see~\cite{grothendieck:esquisse}; the corresponding combinatorial
objects for the modular surface $\mathfrak h/\PSL_2(\Z)$ are called
``Conway diagrams'',
see~\cite{atkin-sd:noncongruence}*{\S3.4}. Methods of constructing
them are addressed, \emph{inter alia},
in~\cites{couveignes-granboulan:dessins,couveignes:exemples,bowers-stephenson:dessins}.

In this section, we consider the case $d\le3$ which can completely be
solved. If $\#Q=2$, then as we said above we may choose
$Q=\{\infty,0\}$ and $f(z)=z^d$. If $d=\#Q=3$ then we may choose
$Q=\{\infty,0,1\}$. Without loss of generality, we may assume that all
points of $Q$ are branched values, since otherwise we are reduced to
the case $\#Q=2$. Up to permutation of the points in $Q$ and the
indices, the only possible triple of permutations is
\[\sigma_1=(1,2,3),\quad \sigma_2=(1,2)\quad\text{and}\quad \sigma_3=(2,3).\] 
To find an analytic model, we seek a rational map $f$ of degree $3$
such that
\[\infty\overset{3:1}\mapsto \infty,\quad 0\overset{2:1}\mapsto0\quad\text{and}\quad 1\overset{2:1}\mapsto1.\] 
This implies $f(z)=3z^2-2z^3$ as the only realization.

The next case we consider is $d=3$ and $\#Q=4$. Using M\"obius
transformations, we may normalise $Q$ to be $\{\infty,0,1,w\}$. Up to
conjugation in $\sym 3$ we may take the first permutations to be
$\sigma_1=\ldots=\sigma_i=(1,2)$. The condition that the permutations
generate a transitive group imply that one of them is not $(1,2)$. Up
to conjugation, we may assume that the first permutation which is not
$(1,2)$ is $\sigma_{i+1}=(2,3)$. Since $\sigma_1\cdot\sigma_2\cdot
\sigma_3\cdot\sigma_4={\rm id}$, this gives four possibilities,
namely, writing $\sigma=(\sigma_1,\sigma_2,\sigma_3,\sigma_4)$,
\begin{xalignat*}{2}
  \sigma&=\bigl((1,2),(1,2),(2,3),(2,3)\bigr), & \sigma&=\bigl((1,2),(2,3),(1,2),(1,3)\bigr),\\
  \sigma&=\bigl((1,2),(2,3),(1,3),(2,3)\bigr), & \sigma&=\bigl((1,2),(2,3),(2,3),(1,2)\bigr).
\end{xalignat*}
To find the corresponding $f$, assume without loss of generality that
$f$ maps $\infty\mapsto \infty$, $0\mapsto 0$, $1\mapsto 1$
and $v\mapsto w$. This forces the map $f$ to have the form
\[f_a(z)=z^2\frac{a(z-1)+1}{(a+2)(z-1)+1},\] for some parameter $a$
subject to $(a+1)(a-1)^3=w a(a+2)^3$; then $(a+1)(a-1)=v a(a+2)$. Since
$w\neq0,1$, the equation defining $a$ in terms of $w$ has four
distinct roots, leading to four candidate maps $f_a$. There is a
bijection between the maps $f$ and the triples of permutations above,
but no canonical one --- it will depend on the specific choice of
$\#Q$ generators $\hat\gamma_i$ of $\pi_1(\Pone(\CC)\setminus
Q,*)$.  Note also that these four solutions are part of a single
Hurwitz class.

\subsection{Reddite C\ae sare}
Various methods have already been considered for the computation of
branched coverings, at least under some restrictions on the
data. Note, first, that a head-on approach, solving numerically the
equations after having converted them to a Gr\"obner basis, works only
for the most simple examples, and in particular is completely
unrealistic for the degree-13 example described
in~\S\ref{ss:dynamics}.

In case $k=3$ and $Q=\{0,1,\infty\}$, the covering is called a
\emph{Belyi map}; if furthermore $\sigma_3$ is a $d$-cycle, then the
covering is called a \emph{Belyi polynomial}. The explicit
construction of Belyi maps has been addressed by numerous
authors. Couveignes and Granboulan describe
in~\cites{couveignes-granboulan:dessins} a method based on writing
Puiseux series for the solution, after having made initial guesses on
the positions of the roots; they obtain in this manner very
high-precision approximations of the co\"efficients of the map, which
allow the determination of their minimal field of definition (they
credit the idea to Oesterl\'e).

Matiyasevich conducted in~\cite{matiyasevich:chebyshev} some
experiments, and showed that Belyi polynomials can be efficiently
computed by an iterative process, increasing the polynomial degree and
adjusting the critical values by Newton's method. The idea is to
iteratively deform the polynomial $z^d$ so as to obtain arbitrary
critical values.

A much more efficient approach has been developed recently by Marshall
and Rohde~\cite{marshall-rohde:convergence}, and is based on the
zipping algorithm~\cite{kuehnau:interpolation}.  Zipping is much
faster, and lets one construct Belyi maps of very high degree. In
particular, Marshall and Rohde managed to describe all Belyi
polynomials of degree $\le14$. They have been able to reproduce the
computations in this article using their method.

\section{Overview of the algorithm}\label{ss:algo}
We are given a list $\sigma_1,\dots,\sigma_k$ of permutations in $\sym
d$ with product $\sigma_1\cdots\sigma_k=1$, and points
$Q_1,\dots,Q_k\in\Pone(\CC)$.

Let $\alpha_i = (\alpha_{i,1},\dots,\alpha_{i,\ell_i})$ be the cycle
lengths of $\sigma_i$; we have $\sum_j\alpha_{i,j}=d$ for all $i$, and
$\sum_{i,j}(\alpha_{i,j}-1)=2d-2$. In the first part of the algorithm,
we enumerate all rational maps with critical values
$Q_1,\dots,Q_k$ such that the multiplicities of the preimages of $Q_i$
are $\alpha_{i,1},\dots,\alpha_{i,\ell_i}$. In the second part, we
select the appropriate rational map among these candidates.

The approach in the first part of the algorithm seems to originate in
Malle~\cite{malle:primitive}; see
also~\cite{malle-matzat:realizierung}.

For the sake of describing its workflow more clearly, the actual
algorithm (described in the remainder of the text) has been slightly
simplified.

\begin{description}
\item[Normalization] Without loss of generality, we assume
  $Q_1=\infty$, $Q_2=0$ and $Q_3=1$. We approximate the other $Q_i$ by
  $\tilde Q_i\in\Pone(\Qb)$. The rational map we seek will leave
  $Q_1,Q_2,Q_3$ fixed. Using this normalization, if all $Q_i$ are
  algebraic then the co\"efficients of the map will also be algebraic.
\item[Finite field solution] We pick a prime $p$, such that the points
  $\tilde Q_i$ have distinct realizations $\overline
  Q_i\in\Pone(\FF_p)$.  We then list all degree-$d$ rational maps
  $\overline F$ over $\FF_p$ with poles and zeroes of multiplicities
  $\alpha_1$ and $\alpha_2$ respectively, and by brute force check for
  each $\overline F$ whether $\overline F-\overline Q_i$ has zeroes of
  multiplicities $\alpha_i$ for all $i \ge 3$.  Note that the rational
  map $\overline F$ is a solution to our original problem over
  $\FF_p$. (If there are no solutions, we restart with a different
  prime $p$).
\item[\boldmath $p$-adic solution] Write $\overline F=\overline
  W_2/\overline W_1$ with $\overline W_2$ monic of degree $d$, and
  $\overline W_1$ of degree less than $d$. (In fact, we later write
  the denominator as $\lambda W_1$ with $W_1$ monic. The present
  discussion uses a simplified notation.) For $i\ge 3$, let $\overline
  W_i=\overline W_2-\overline Q_i\overline W_1$ be the numerator of
  $\overline F-\overline Q_i$.  We compute high-precision $p$-adic
  approximations $\hat Q_i$ of the $\tilde Q_i$, and lift each
  $\overline W_i$ from $\FF_p$ to a high-precision polynomial $\hat W_i$
  over $\ZZ_p$, in such a manner that we have $\hat W_i=\hat
  W_2-\hat Q_i\hat W_1+O(p^N)$ for large $N$. This lifting can be done
  by Hensel's lemma, because by Corollary~\ref{cor:buff}, the Jacobian
  of the system $\{W_i=W_2-Q_i W_1\}$ is invertible at a solution for
  almost every prime $p$. (If $D\overline F$ happens not to be
  invertible, we restart with a different prime).
\item[Algebraic solution] Using the lattice-reduction algorithm
  LLL~\cite{lenstra-l-l:factoring}, we find polynomials $W_i$ over
  $\Qb$, with co\"efficients of small height (small degree and
  co\"efficients of minimal polynomial) that are close to $\hat W_i$
  obtained at the previous step. Using exact arithmetic over $\Qb$, we
  check that the solution $W_2/W_1$ is correct. (If not, we either
  compute a finer $p$-adic approximation, or higher-degree algebraic
  number approximations, or we restart altogether with a larger
  prime).
\item[Complex solution] For each co\"efficient $c_{i,j}\in\Qb$ of
  $W_i$, given by its minimal polynomial over $\Q$, we compute (to
  high, user-specified precision) all the roots $\tilde c_{i,j,s}$ of
  its minimal polynomial, as floating-point complex numbers. Not all
  choices of $\tilde c_{i,j,s}$ are compatible: there may exist some
  extra constraints between one co\"efficient and another (such as,
  for example, that they are complex conjugates of each other). We
  determine these extra constraints as follows: we choose small,
  random integers $m,n$, compute the minimal polynomial of $m
  c_{i,j}+n c_{i',j'}$, and compute (again to high precision) its
  roots $\tilde d_{i,j,i',j',t}$. We then pair together those roots
  $(\tilde c_{i,j,s},\tilde c_{i',j',s'})$ for which $\tilde
  d_{i,j,i',j',t} \approx m\tilde c_{i,j,s}+n\tilde c_{i',j',s'}$ for
  some $t$. By considering enough of these pairs we can stitch
  together a collection of compatible co\"efficient approximations
  $\tilde c_{i,j,s}$ embracing all $i,j$.
 
  We call $\widetilde C$ the collection of all co\"efficients $\tilde
  c_{i,j,s}$, and note that the rational map is determined by its
  zeroes, its poles, and the normalization condition that $Q_1=1$ is
  fixed. Since $\{0,\infty\}\subset Q$, these zeroes and poles are
  determined by $\widetilde C$.
\end{description}

The second step of the algorithm checks, by path lifting, that the
monodromy around $Q_i$ is correct. For each of the Galois conjugate
solutions $(\widetilde C,\widetilde f)$ obtained in the first step, we
do the following:

\begin{description}
\item[Triangulate] We are given a floating-point approximation
  $\widetilde Q$ of $Q$. We compute a triangulation $\mathscr Q$ of
  $\Pone(\CC)$ whose vertex set contains $\widetilde Q$, and a
  triangulation $\mathscr C$ of $\Pone(\CC)$ whose vertex set
  contains $\widetilde C$. For efficiency reasons, we use
  \emph{Delaunay triangulations}, see~\S\ref{ss:mono}. We compute the
  dual triangulation $\mathscr Q^\perp$; it has one vertex per face of
  $\mathscr Q$, and edges transverse to those of $\mathscr Q$. We fix
  a vertex $*\in\mathscr Q^\perp$ as our basepoint.
\item[Lift the triangulation] Let $W$ denote the vertices of $\mathscr
  Q^\perp$. For each $w\in W$, we number arbitrarily $w_1,\dots,w_d$
  the $\widetilde f$-preimages of $w$.

  For each edge $\varepsilon\in\mathscr Q^\perp$, going from $w'$ to
  $w''$, we compute a permutation $\varsigma_\varepsilon\in\sym d$
  such that the $\widetilde f$-lift of $\varepsilon$ starting at
  $w'_i$ ends at $w''_{\varsigma_\varepsilon(i)}$. There are two
  strategies for this, one is by subdividing appropriately the path
  $\varepsilon$ and playing ``connect-the-dots'', the other uses more
  efficiently the triangulation $\mathscr C$.
\item[Read permutations] For each critical value $Q_i\in Q$, let
  $\varepsilon(1),\dots,\varepsilon(n)$ be the sequences of edges
  traversed by a path in $\mathscr Q^\perp$ that starts and ends in
  the basepoint $*$, and surrounds once counterclockwise the point
  $Q_i$ and no other vertex of $Q$. Compute the permutation
  $\sigma'_i=\varsigma_{\varepsilon(1)}\cdots\varsigma_{\varepsilon(n)}$.
\item[Check] The data $(C,f)$ are a valid solution to the Hurwitz
  problem if and only if there exists a permutation $\tau\in\sym d$
  such that $\sigma'_i=(\sigma_i)^\tau$.
\end{description}

\subsection{Implementation}
The fourth-named author has implemented the first part of the
algorithm, mainly in C, and the first-named author has implemented the
second part of the algorithm, mainly in
\textsc{GAP}~\cite{gap4:manual}. By far the most time-consuming part
of the procedure is the search for a solution over a finite
field. Example~\ref{ex:search} required approximately 15 minutes on a
desktop, 30-SPECint2006 computer.  The code is maintained by the
fourth-named author, and is available at
\[\texttt{https://github.com/jakobkroeker/HMAC}\]

\section{The space of rational maps}
We show, in this section, that (as soon as the prime $p$ is
sufficiently large) we may lift every $\FF_p$-solution to $\ZZ_p$.
This follows from the well known fact that the Hurwitz spaces are
smooth. We could not find the precise statement we need in the
literature, so we give a complete proof.

Let $d\geq 2$ be an integer and denote by $\Rat_d$ the space of
rational maps of degree $d$, which may be identified with a Zariski
open subset of $\mathbb P^{2d+1}(\CC)$.

Let $k\geq 3$ be an integer and let $F:\sphere\to \sphere$ be a
ramified covering branched over $Q=\{Q_1,\ldots,Q_k\}$. Note that,
according to the Riemann-Hurwitz Formula, $C:=F^{-1}(Q)$ contains
exactly $(k-2)d+2$ points.  We write $C=\bigcup_i C_i$ with
$C_i:=F^{-1}\{Q_i\}=\{C_{i,1},\ldots,C_{i,\ell_i}\}$, and for each
$j\in \{1,\ldots, \ell_i\}$ we let $\alpha_{i,j}$ be the local degree
of $F$ at $C_{i,j}$.

Let $\mathfrak Q$ be the smooth quasiprojective variety of injective
maps $\mathfrak q:Q\to \Pone(\CC)$. For $\mathfrak q\in
\mathfrak Q$, we use the notation $q_i:=\mathfrak q(Q_i)$. Similarly,
let $\mathfrak C$ be the smooth quasiprojective variety of injective
maps $\mathfrak c:C\to \Pone(\CC)$. For $\mathfrak c\in
\mathfrak C$, we use the notation $c_{i,j}:=\mathfrak c(C_{i,j})$.
The quasiprojective variety $\mathfrak Y:=\mathfrak C\times \mathfrak
Q$ is smooth. We shall prove that the subvariety
\[\mathfrak X:=\bigl\{(\mathfrak c,\mathfrak q)\in \mathfrak Y ~\mid~(\exists f\in \Rat_d)~(\forall i,j)~f(c_{i,j}) = q_i\text{ and } {\rm deg}_{c_{i,j}} f = \alpha_{i,j}\bigr\}\]
is also smooth, and regularly parametrised:

\begin{prop}\label{prop:buff}
  The variety $\mathfrak X$ is smooth of dimension $k+3$, locally
  regularly parametrised by $(q_1,\dots,q_k,c_{1,1},c_{2,1},c_{3,1})$.
\end{prop}

Observe that, for $(\mathfrak c,\mathfrak q)\in \mathfrak X$, there is
a unique rational map $f\in \Rat_d$ such that $f(c_{i,j}) = q_i$ and
${\rm deg}_{c_{i,j}} f = \alpha_{i,j}$ for all $(i,j)$. Indeed,
knowing a rational map above three points completely determines the
rational map (it is even enough to know the full preimage of two
points plus one preimage of a third point).

Note that the group of M\"obius transformations acts on $\mathfrak C$
and $\mathfrak Q$ by postcomposition:
\[(M,N)\cdot (\mathfrak c,\mathfrak q) := (M\circ \mathfrak c, N\circ
\mathfrak q).
\]
The quotient space may be identified with $\mathfrak Y_0:=\mathfrak
C_0\times \mathfrak Q_0$ with
\[\mathfrak C_0:=\bigl\{\mathfrak c\in \mathfrak C~\mid ~ c_{1,1}=\infty, ~c_{2,1}=0\text{ and } c_{3,1} = 1\bigr\}\]
and
\[\mathfrak Q_0:=\bigl\{\mathfrak q\in \mathfrak Q~\mid ~ q_1=\infty, ~q_2=0\text{ and } q_3 = 1\bigr\}.\]
The projection $\mathfrak Y\to \mathfrak
Y/(\PSL_2(\CC)\times\PSL_2(\CC))\cong \mathfrak Y_0$ is a submersion.

The action preserves $\mathfrak X$ as indicated on the following
commutative diagram:
\[\diagram
C\ar@{^{(}->}[r]^{\mathfrak c}\dto_F & \Pone(\CC)\dto^f \rto^M & \Pone(\CC)\dto^{N\circ f\circ M^{-1}} \\
Q\ar@{^{(}->}[r]^{\mathfrak q} & \Pone(\CC) \rto^N & \Pone(\CC).
\enddiagram
\]
It is therefore enough to show that $\mathfrak X_0:=\mathfrak X\cap
\mathfrak Y_0$ is a smooth subvariety of $\mathfrak Y_0$ locally
regularly parametrised by $(q_4,\dots,q_k)$.

We first write equations for $\mathfrak X_0$. 
To each $(\mathfrak c,\mathfrak q,\lambda)\in \mathfrak C_0\times \mathfrak Q_0\times \CC^*$, we associate a collection of monic polynomials 
$(W_i)_{i\in \{1,\ldots,k\}}$ defined by
\[W_1(z):=\prod_{j=2}^{\ell_1} (z-c_{1,j})^{\alpha_{1,j}}
\quad\text{and for }i\geq 2\quad W_i(z):=\prod_{j=1}^{\ell_i} (z-c_{i,j})^{\alpha_{i,j}}\]
and a collection of rational maps $(f_i)_{i\in \{2,\ldots,k\}}$ defined by
\[f_i:=\frac{W_i}{\lambda W_1}+q_i.
\]
Note that these are degree-$d$ rational maps with poles of order
$\alpha_{1,j}$ at $c_{1,j}$. In addition, $f_i$ maps $c_{i,j}$ to
$q_i$ with local degree $\alpha_{i,j}$. It follows that $(\mathfrak
c,\mathfrak q)\in \mathfrak X_0$ if and only if there is a $\lambda\in
\C^*$ such that $f_i=f_2$ for all $i\in \{3,\ldots, k\}$, that is,
$F_i=0$ with
\begin{equation}\label{eq:buff}
  F_i:=W_i+\lambda q_i W_1-W_2.
\end{equation}
In that case, we use the notation
\[f_{(\mathfrak c,\mathfrak q,\lambda)}:=f_2=f_3=\cdots=f_k.\]

In other words, consider the map
\[\mathcal F:=(F_3,\ldots,F_k):\mathfrak Y_0\times \CC\to (\CC[z]_{\deg<d})^{k-2}.\]
Then, $(\mathfrak c,\mathfrak q)\in \mathfrak X_0$ if and only if
there is a $\lambda \in \C^*$ such that $\mathcal F(\mathfrak
c,\mathfrak q,\lambda)=0$.

According to the following Lemma and the Implicit Function Theorem,
the subvariety of $\mathfrak Y_0\times \C$ defined by the equation
$\mathcal F=0$ is smooth of dimension $k-3$, locally regularly
parametrised by $(q_4,\dots,q_k)$. It follows that its projection to
the $\mathfrak Y_0$ component, namely $\mathfrak X_0$, is also smooth of
dimension $k-3$, locally regularly parametrised by $(q_4,\dots,q_k)$.

\begin{lem}\label{lemma:implicit}
  If $\mathcal F(\mathfrak c,\mathfrak q,\lambda)=0$, then the
  derivative $D_{(\mathfrak c,\mathfrak q,\lambda)}\mathcal F$
  restricts to an isomorphism $T_\mathfrak c \mathfrak C_0\times
  \{0\}\times T_\lambda \C\to T_0(\CC[z]_{\deg<d})^{k-2}$.
\end{lem}

We postpone the proof of the lemma to Section~\ref{sec:prooflemma} and
mention immediately a corollary that we shall use later. For
$\mathfrak q\in \mathfrak Q$, let $\mathcal F_{\mathfrak q}:\mathfrak
C_0\times \CC^*\to (\CC[z]_{\deg<d})^{k-2}$ be defined by
\[\mathcal F_{\mathfrak q}(\mathfrak c,\lambda):=\mathcal F(\mathfrak c,{\mathfrak q},\lambda).\]

\begin{cor}\label{cor:buff}
  Assume that $(\mathfrak c,\mathfrak q,\lambda)$ is defined over
  $\Qb$ with $\mathcal F(\mathfrak c,{\mathfrak q},\lambda)=0$.  Then,
  for almost every prime $p$, the derivative $D\mathcal F_{\mathfrak
    q}$ at $(\mathfrak c,\lambda)$ is invertible mod $p$.
\end{cor}

\begin{proof}
  Since the point $(\mathfrak c,\mathfrak q,\lambda)$ is defined over
  $\Qb$, its co\"ordinates may be written using algebraic integers,
  and reduced mod $p$. For all except finitely many values of $p$, the
  resulting reduction gives a genuine point, namely where the
  reductions of $(\mathfrak c,\mathfrak q)$ are injective and the
  reduction of $f_{(\mathfrak c,\mathfrak q,\lambda)}$ has degree $d$.
  Since $D\mathcal F$ is invertible over $\Qb$, it may be written as
  $a/N$ for a matrix $a$ with algebraic integer entries and $N\in\NN$;
  then the reduction modulo $p$ of $D\mathcal F$ is invertible for all
  primes not dividing $N$.
\end{proof}

Before embarking in the proof of the Lemma, we first build up a
description of the tangent space of $\Rat_d$. 

\subsection{The tangent space to rational maps}
Consider a rational map $f\in \Rat_d$. A tangent vector to $f$ is
\[\dot f:=\frac{{\mathrm d}f_t}{{\mathrm d}t}\Big|_{t=0}\] 
for a holomorphic family of rational maps $(f_t)_{|t|<\epsilon}$ with
$f_0=f$.

For every $z\in\Pone(\CC)$, the vector $\dot f(z)$ is a tangent
vector in $\Pone(\CC)$ at $f(z)$; in other words, $\dot f$ is a
section of the pullback bundle $f^*\bigl(T\Pone(\CC)\bigr)$. It
can be pulled back to a vector field on $\Pone(\CC)$, as
\[\eta(z)=-(D_z f)^{-1}(\dot f(z)),\]
or, in co\"ordinates, $\eta(z)=-\dot f(z)/f'(z)$. Therefore, $\eta(z)$
is a meromorphic vector field on $\Pone(\CC)$, holomorphic away
from critical points of $f$, and with a pole of order at most $m$ at
critical points of multiplicity $m$.

Geometrically, $\eta(z)$ is the movement at time $t=0$ of the point
$z_t=f_t^{-1}(f(z))$. This point $f_t$ can be followed away from
critical points, by the Implicit Function Theorem.

%

We consider now local perturbations of $f$ at a critical point, i.e.\
we assume that the vector field $\dot f$ is given by a path $f_t =
\phi_t \circ f \circ \psi_t^{-1}$ with $\phi_t,\psi_t$ analytic
perturbations of the identity at the critical value and point $v,c$
respectively of $f$.

Let $c_t$ denote the critical point of $f_t$ and let $v_t$ denote its
critical value; then $c_t=\psi_t(c)$ and $v_t=\phi_t(v)$. Let $\dot c$
denote the motion vector of $c_t$, and let $\dot v$ denote the motion
vector of $v_t$. Then $\dot c = \dot\psi(c)$ and $\dot
v=\dot\phi(v)$. Now $\dot f= \dot\phi\circ f-D f \circ\dot\psi$,
because $\phi_0=\psi_0=\operatorname{id}$. Therefore,
\[\eta + f^*\dot\phi = \dot\psi.\]
At $v$, the vector field $\dot\phi$ takes value $\dot v$, the vector
field $\eta+f^*\dot\phi$ is holomorphic at $c$, and its constant term
is $\dot c$. If $\dot v=0$, then $f^*\dot \phi$ is holomorphic near $c$ 
and vanishes at $c$. 

Therefore, whenever we have a family of rational maps $(f_t)$ for
which we can follow a critical point $c_t$ and its associated critical
value $v_t$ with $\dot v=0$, the vector field $\eta$ is holomorphic
near $c$ and coincides with $\dot c$ at $c$.

\subsection{Proof of Lemma \ref{lemma:implicit}\label{sec:prooflemma}}

For $i\in \{3,\ldots k\}$, we have 
\[f_i-f_2=\frac{F_i}{\lambda W_1}.
\]
Recall that if $\mathcal F(\mathfrak c, \mathfrak q, \lambda)=0$, then
$f_i=f_2$ for all $i\in \{3,\ldots, k\}$. We denote by $f$ this common
rational map of degree $d$. If in addition$(\dot {\mathfrak c}, \dot
{\mathfrak q}, \dot \lambda)$ belongs to the Kernel of $D\mathcal F$
at $(\mathfrak c, \mathfrak q, \lambda)$, then
\[\dot f_i-\dot f_2 = \frac{\dot F_i}{\lambda W_1} - \frac{F_i\cdot (\dot \lambda W_1+\lambda \dot W_1)}{(\lambda W_1)^2} = 0,\]
so $\dot f_i = \dot f_2$ for all $i\in \{3,\ldots, k\}$. We denote by
$\dot f$ this common tangent vector to $\Rat_d$ at $f$ and by $\eta$
the corresponding meromorphic vector field on $\Pone(\CC)$.

As $(\mathfrak c,\mathfrak q,\lambda)$ varies in $\mathfrak Y_0\times
\CC^*$, the $f_i$-preimages of the points $q_1=\infty$ and $q_i$ vary
holomorphically: they are the points $c_{1,j}$ and $c_{i,j}$.
According to the previous remark, we see that if $\dot{ \mathfrak
  q}=0$ then, for all $i\in \{2,\dots, k\}$, the meromorphic vector
field $\eta$ is holomorphic near $c_{1,j}$, coincides with $\dot
c_{1,j}$ at $c_{1,j}$, and furthermore is holomorphic near $c_{i,j}$
and coincides with $\dot c_{i,j}$ at $c_{i,j}$.

Thus, $\eta$ is a holomorphic vector field on the whole sphere
$\Pone(\CC)$ and coincides with  $\dot c_{i,j}$ at $c_{i,j}$. In
particular, it vanishes at $c_{1,1}=\infty$, $c_{2,1}=0$ and
$c_{3,1}=1$. A holomorphic vector field with at least $3$ zeroes
globally vanishes. Therefore, $\eta=0$ and $\dot c_{i,j}=0$ for all
$(i,j)$.  In addition, for all $i\in \{1,\ldots, k\}$, we have that
$\dot W_i=0$ and for all $i\in \{3,\ldots k\}$, we have that $0=\dot
F_i = \dot \lambda q_i W_1.$ This shows that $\dot \lambda=0$.

Let us summarize: if $\mathcal F(\mathfrak c, \mathfrak q,
\lambda)=0$ and if $(\dot {\mathfrak c},0,\dot \lambda)$ belongs to
the Kernel of $D\mathcal F$ at $(\mathfrak c, \mathfrak q, \lambda)$,
then $\dot {\mathfrak c}=0$ and $\dot \lambda = 0$. So, the
restriction of $D_{(\mathfrak c, \mathfrak q, \lambda)}\mathcal F$ to
$T_\mathfrak c \mathfrak C_0\times \{0\}\times T_\lambda \C$ is
injective. Since $T_\mathfrak c \mathfrak C_0\times \{0\}\times
T_\lambda \C$ and $T_0(\CC[z]_{\deg<d})^{k-2}$ have the same
dimension, that is $(k-2)d$, this restriction is an isomorphism as
required.

\section{Finding a solution in a finite field}
We describe in this section an efficient method of finding a rational
function over a finite field with prescribed critical values and
multiplicities.

We start by recalling some facts about univariate polynomials over
non-algebraically-closed fields $\Bbbk$ of arbitrary
characteristic. For this we need some notation:

\begin{notation}
  An ordered sequence $\alpha = (\alpha_1,\dots,\alpha_k)$ with
  $\alpha_1 \ge \dots \ge \alpha_k>0$ and $\sum\alpha_i = d$ is called a
  \emph{partition} of $d$. With the shorthand notation
  \[\beta^\mu := (\underbrace{\beta,\dots,\beta}_{\text{$\mu$ times}})
  \]
  we can always write $\alpha = (\alpha_1,\dots,\alpha_k) =
  (\beta_1^{\mu_1}, \dots , \beta_n^{\mu_n})$ with $\beta_1 >
  \dots > \beta_n$ and appropriate $\mu_i$. For example,
  $13=4+3+2+2+2$ is written as $\alpha=(4,3,2,2,2)=(4^1,3^1,2^3)$.

  The partition $\alpha^*$ defined by $\alpha_j^* := \#\{i \colon
  \alpha_i \ge j\}$ is called the \emph{dual partition} of
  $\alpha$. For example, $\alpha^*=(5,5,2,1)$.

  Let $f \in \Bbbk[x]$ be a degree-$d$ polynomial, let $l_i = (x-C_i)
  \in \overline\Bbbk[x]$ be its distinct linear factors over an
  algebraic closure of $\Bbbk$, and let $\alpha_i$ be their
  multiplicities, so that
  \[f = \prod_{i=1}^k l_i^{\alpha_i}.\]
  Without restriction we can assume $\alpha_1 \ge \dots \ge \alpha_k$
  and $\alpha = (\alpha_1,\dots,\alpha_k)$ is a partition of $d$. In
  this situation we say that $f$ is of \emph{shape} $\alpha$.

  If we write $\alpha = (\alpha_1,\dots,\alpha_k) =
  (\beta_1^{\mu_1}, \dots , \beta_n^{\mu_n})$ as above, we can
  write
  \[
  f = \prod_{i=1}^n f_i^{\beta_i}
  \]
  with $\deg f_i = \mu_i$ and $f_i$ the product of those linear
  forms that have multiplicity $\beta_i$. In this situation the $f_i$
  are coprime.
\end{notation}

\begin{lem}\label{lgcdone}
  Let $\Bbbk$ be a field of characteristic $p$, let $f\in \Bbbk[x]$ be
  a univariate polynomial of shape $\alpha$ and write $f =
  \prod_{i=1}^k l_i^{\alpha_i}$ with $l_i \in \overline\Bbbk[x]$.  If $\alpha_i
  < p$ for all $i$ then
  \[
  \gcd(f,f') = \prod_{i=1}^k l_i^{\alpha_i-1}.
  \]
\end{lem}
\begin{proof}
  We have
  \[
  f' = \sum_{i=1}^k\alpha_i\frac{f}{l_i}l_i' = \Bigl(\prod
  l_i^{\alpha_i-1}\Bigr) \sum_{i=1}^k \alpha_i\Bigl(\prod_{j\not=i}
  l_j\Bigr)l_i'.
  \]
  This shows that $\prod_{i=1}^k l_i^{\alpha_i-1}$ divides the
  $\gcd$. Assume now that there is another linear factor $l$ in the
  $\gcd$. Since the $\gcd$ divides $f$ there exists an index $s$ with
  $l=l_s$. Since the $\gcd$ divides $f'$ we have that $l$ divides
  \[
  \sum_{i=1}^k \alpha_i\Bigl(\prod_{j\not=i} l_j\Bigr)l_i'.
  \]
  Now all summands except for $\alpha_s\Bigl(\prod_{j\not=s}
  l_j\Bigr)l_s'$ are divisible by $l$. Since $\alpha_s$ and $l_s'$ are
  nonzero in $\Bbbk$ and $\overline\Bbbk[x]$ respectively, it follows that $l$
  must divide $\prod_{j\not=i} l_j$. This is impossible since the
  $l_s$ are pairwise coprime.
\end{proof}

\begin{cor} \label{cgcdall}
  With the notations of the previous Lemma we  have
  \[
  \gcd(f,f',\dots,f^{(e)}) = \prod_{i \colon \alpha_i>e}
  l_i^{\alpha_i-e}.
  \]
\end{cor}
\begin{proof}
  Lemma \ref{lgcdone} and induction.
\end{proof}

\begin{cor}  \label{cDualShape}
  With the notations above let $\alpha^*$ be the dual partition of $\alpha$. Then
  \[
  \alpha_e^* = \deg \gcd(f,f',\dots,f^{(e-1)}) - \deg
  \gcd(f,f',\dots,f^{(e)}).
  \]
\end{cor}
\begin{proof}
  We have
  \[
  g_e := \frac {\gcd(f,f',\dots,f^{e-1})} {\gcd(f,f',\dots,f^{e})} =
  \frac { \prod_{i \colon \alpha_i>e-1} l_i^{\alpha_i-e+1}} { \prod_{i
      \colon \alpha_i>e} l_i^{\alpha_i-e}} = \prod_{i \colon \alpha_i
    \ge e} l_i.
  \]
  It follows that
  \[
  \deg g_e = \deg \prod_{i \colon \alpha_i \ge e} l_i  = \#\{i \colon \alpha_i \ge e\} = \alpha_e^*.\qedhere
  \]
\end{proof}

\begin{alg}[Compute the shape of a polynomial]\label{aShape}\rule{0mm}{0mm}\\
  \underline{Given:} a polynomial $f\in\Bbbk[x]$\\
  \underline{Return:} the shape $(\alpha_1,\dots,\alpha_k)$ of $f$.

  Write $d=\deg(f)$. For each $e=0,\dots,d$, compute
  $g_e=\gcd(f,f',\dots,f^{(e)})$. For each $e=1,\dots,d$ define then
  $\alpha^*_e=\deg(g_{e-1})-\deg(g_e)$. Return the dual of the partition
  $(\alpha^*_1,\dots,\alpha^*_d)$.
\end{alg}
\begin{proof}[Proof of validity]
  This directly follows from Corollary~\ref{cDualShape}.
\end{proof}

We may collect linear factors of the same multiplicity, so as to avoid
field extensions:
\begin{cor}\label{cRational}
  With the notation of Lemma \ref{lgcdone} choose $\beta_1 > \dots >
  \beta_n$ among the $\alpha_i$ such that $f= \prod_{i=1}^n
  f_i^{\beta_i}$ with $f_j = \prod_{i\colon \alpha_i = \beta_j}
  l_i$. Then the $f_i$ are defined over $\Bbbk$.
\end{cor}
\begin{proof}
  Since the calculation of a $\gcd$ does not require field extensions,
  we have
  \[
  \gcd(f,f',\dots,f^{(e)}) \in \Bbbk[x].
  \]
  By Corollary \ref{cgcdall} we then have
  \[
  g_j := \frac {\gcd(f,f',\dots,f^{(\beta_j-1)})}
  {\gcd(f,f',\dots,f^{(\beta_j)})} = \frac { \prod_{i \colon
      \beta_i>\beta_j-1} f_i^{\beta_i-\beta_j+1}} { \prod_{i \colon
      \beta_i>\beta_j} f_i^{\beta_i-\beta_j}} = \prod_{i=1}^j f_i \in
  \Bbbk[x]
  \]
  so
  \[
  f_j = \frac{g_j}{g_{j+1}} \in \overline\Bbbk[x]\cap\Bbbk(x)=\Bbbk[x].\qedhere
  \]
\end{proof}

We are now ready to describe our algorithm searching for rational maps
over $\FF_p$.

\begin{alg}[Compute all rational maps over $\Bbbk$ with given shape above given points]\label{aOne}\rule{0mm}{0mm}\\
  \underline{Given:} a finite field $\Bbbk$, a list of points
  $\overline Q_1,\dots,\overline Q_k\in\Pone(\Bbbk)$, an integer
  $d$, and a list of partitions $(\alpha_1,\dots,\alpha_k)$ of $d$
  with $\alpha_i=(\alpha_{i,1},\dots,\alpha_{i,\ell_i})$ satisfying
  $\sum_i\sum_j(\alpha_{i,j}-1)=2d-2$\\
  \underline{Return:} all rational maps over $\Bbbk$ of degree $d$
  such that every $\overline Q_i$ has $\ell_i$ preimages with local
  degrees $\alpha_{i,1},\dots,\alpha_{i,\ell_i}$ respectively.

  We choose a M\"obius transformation $M\in\PSL_2(\Bbbk)$ sending
  $\overline Q_1$ to $\infty$ and $\overline Q_2$ to $0$.

  We write each partition $\alpha_i$ in compacted form as
  $\alpha_i=(\beta_{i,1}^{\mu_{i,1}},\dots,\beta_{i,n_i}^{\mu_{i,n_i}})$.

  We enumerate all $\ell_1$-tuples of monic polynomials
  $(f_1,\dots,f_{n_1})$ with $\deg(f_j)=\mu_{i,j}$, and all
  $\ell_2$-tuples of monic polynomials $(g_1,\dots,g_{n_2})$ with
  $\deg(g_j)=\mu_{2,j}$.

  For each such pair of tuples, we compute
  \[W_1 = \prod_{j=1}^{n_1} f_j^{\beta_{1,j}}\text{ and }
  W_2 = \prod_{j=1}^{n_2} g_j^{\beta_{2,j}}.\]

  Using Algorithm~\ref{aShape}, we filter those $(W_1,W_2)$ such that
  the shape of $W_1$ is $\alpha_1$ and the shape of $W_2$ is
  $\alpha_2$ (this fails only if a pair $f_i,g_j$ is not coprime).

  By computing their g.c.d., we filter those $(W_1,W_2)$ such that
  $W_1$ and $W_2$ are coprime.

  For each $i=3,\dots,k$, let $\Lambda_i\subset\Bbbk$ be the set of
  $\lambda\in\Bbbk$ such that the shape (computed using
  Algorithm~\ref{aShape}) of $W_i:=W_2-\lambda M(\overline Q_i)W_1$ is
  $\alpha_i$. We filter those rational maps for which
  $\bigcap_{i=3}^k\Lambda_i$ is non-empty.

  We return all the rational maps $M^{-1}\circ(W_2/\lambda W_1)$, for
  all $\lambda\in\bigcap_{i=3}^k\Lambda_i$, that survived the
  filtering.
\end{alg}
\begin{proof}[Proof of validity]
  Let first $f:=M^{-1}\circ(W_2/\lambda W_1)$ be a rational map
  returned by the algorithm. For $i=2,\dots,k$, consider the rational
  map $f_i=M^{-1}\circ(W_i/\lambda W_1)$. By the very definition of
  $W_i$ (compare with~\eqref{eq:buff}), we have $f=f_2=\dots=f_k$. On
  the other hand, the $f_i$-preimages of $\overline Q_i$ are the
  zeroes of $W_i$, so they have multiplicities $\alpha_i$.

  On the other hand, let $f$ be a rational map such that every
  $\overline Q_i$ has $\ell_i$ preimages with local degrees
  $\alpha_{i,1},\dots,\alpha_{i,\ell_i}$ respectively. Then, for every
  M\"obius transformation $M$ sending $\overline Q_1$ to $\infty$ and
  $\overline Q_2$ to $0$, the rational map $M\circ f$ will be of the
  form $W_2/\lambda W_1$, for monic polynomials $W_1,W_2$ of
  respective shapes $\alpha_1,\alpha_2$ and a scalar
  $\lambda\in\Bbbk$. Furthermore, by Corollary~\ref{cRational}, both
  $W_1$ and $W_2$ factor over $\Bbbk[x]$ into polynomials of degrees
  $\mu_{1,j}$ and $\mu_{2,j}$ respectively. The rational map
  $M\circ f-\overline Q_i$ will be of the form $W_i/\lambda W_1$ with
  $W_i$ of shape $\alpha_i$; therefore, that solution $f$ will be
  returned by the algorithm.
\end{proof}

\begin{rem}
  Algorithm~\ref{aOne} is the most computationally-intensive part of
  our procedure. Its performance is improved in the following ways:
  \begin{enumerate}
  \item If $\mu_{i,j} = 1$ for some $i,j$, then we may assume,
    after permuting the shapes $\alpha_i$, that $i=1$ so that $W_1$
    contains a power of a linear factor $f_i^{\beta_{i,j}}$. Fixing
    the corresponding preimage of $\infty$ to be $\infty$ amounts to
    the choice $f_i=1$, so that the degree of $W_1$ is actually
    $d-\beta_{i,j}$. This speeds up the search by a factor $p$.

    Similarly, if up to permutation of the indices there are more
    $\mu_{i,j}=1$, with $i\in\{1,2\}$, then the corresponding
    factors may be assumed to be $x$ and $x-1$.

    On the other hand, if all $\mu_{i,j}\ge2$, then no
    normalization of the critical points may be assumed, and in
    particular $\infty$ should not be assumed to be a preimage of some
    $\overline Q_i$.
  \item When using Corollary \ref{cDualShape} one can detect a wrong
    shape already if $\gcd(f,f')$ or for that matter any
    $\gcd(f,\dots,f^{(e)})$ with $e < \alpha_1$ has the wrong
    degree. We stop the calculation of $\gcd$'s as soon as this
    happens. This speeds up the process by a factor of about
    $\alpha_{1}$. Similarly, as soon as the intersection of the
    $\Lambda_i$ already computed is empty, the pair $(W_1,W_2)$ should
    be discarded.
  \item If the largest $\mu_{i,j}$ is small enough (e.g. $7$ or $8$
    in $\FF_{11}$) we can enumerate all monic irreducible
    homogeneous polynomials of degree $\le \mu_{i,j}$ and build the
    $f_i$ and $g_i$ out of them, while taking care that no irreducible
    piece is used twice.  We can then omit checking shape and
    coprimeness of $W_1$ and $W_2$ as these conditions are then
    automatically satisfied.
  \end{enumerate}
\end{rem}

\begin{example}\label{ex:search}
  Over $\FF_{11}$ we searched for a rational map of shape
  $(4,3,2,2,2)$, $(4,3,2,2,2)$ and $(4,3,2,2,2)$; we chose $\overline
  Q_1=\infty$ and $\overline Q_2=0$, and didn't specify $\overline
  Q_3$, letting on the contrary the algorithm determine choose it for
  us. We found the solution
  \[
  \frac{W_2}{W_1} = \frac{x^{4} (x+3)^{3} (x^{3}-3x-5)^{2}}{(x-5)^{3}
    (x^{3}+3x^{2}+2x+3)^{2}}.
  \]
  It has indeed the desired shape as we have the following
  factorisation
  \[W_2+4W_1 = (x-1)^{4} (x-3)^{3} (x^{3}-2x-3)^{2},
  \]
  which implies, for the choice $\overline Q_3=1$, the value
  $\lambda=-4$.
\end{example}

\section{Lifting a solution from $\FF_p$ to $\ZZ_p$}
The lift from $\FF_p$ to $\ZZ_p$ is done using Hensel's lemma (namely,
Newton's method in positive characteristic):
\begin{prop}[Hensel's Lemma]\label{pHensel}
  Let $\mathcal F = (F_1,\dots,F_m)$ be a vector of polynomials, with
  $F_i \in \ZZ[x_1,\dots,x_m]$, and let $J=(\frac{\mathrm
    d F_i}{\mathrm d x_j})$ be the Jacobian matrix of $F$. Assume that
  $a = (a_1,\dots,a_m) \in \ZZ^m$ satisfies
  \[\mathcal F(a) \equiv 0 \mod p^N,
  \]
  that $J(a)$ is invertible modulo $p^N$, and let $J^{-1}(a)$ be an
  inverse modulo $p^N$.  Then
  \[\mathcal F(a+b p^N) \equiv 0 \mod p^{2N},
  \]
  for
  \[b :=  -\frac{F(a)}{p^N}J^{-1}(a).
  \]
  Furthermore $J(a+b p^N)$ is invertible modulo $p^{2N}$.
\end{prop}
\begin{proof}
  $F(a)$ is divisible by $p^N$ since $F(a) \equiv 0 \mod p^N$. Therefore
  $b$ is well defined.  We have
  \begin{align*}
    F(a+p^N b) &\equiv F(a) + b J(a) p^N \mod p^{2N} \\
    &\equiv F(a) -  \frac{F(a)}{p^N}J^{-1}(a) J(a) p^N \mod p^{2N} \\
    &\equiv 0 \mod p^{2N}.
  \end{align*}
  The invertibility holds more generally. Let $A$ and $B$ be matrices
  with $AB \equiv \mathbb1 \mod p^N$. We can then write
  \[AB \equiv \mathbb1 + p^N C \mod p^{2N}.
  \]
  In this situation we have
  \begin{align*}
    A(B-p^NBC) &\equiv \mathbb1+p^N C-p^N A B C \mod p^{2N}\\
    &\equiv \mathbb1 \mod p^{2N}
  \end{align*}
  since $AB \equiv \mathbb1 \mod p^N$; so $B' = B-p^NBC$ is an inverse
  to $A$ modulo $p^{2N}$.
\end{proof}

Consider the following data: a ring $\Bbbk$; a family of polynomials
$W_i\in\Bbbk[x]$ of degree at most $d$, for $i=1,\dots,k$, with
factorisations $W_i=\prod_{j=1}^{n_j}W_{i,j}^{\beta_{i,j}}$; a
parameter $\lambda\in\Bbbk^\times$; and a sequence of points
$Q_1=\infty,Q_2,\dots,Q_k\in\Pone(\Bbbk)$. We say that they are
\emph{coherent} if $W_i/\lambda W_1+Q_i$ is independent of
$i=2,\dots,k$.

We say that they are \emph{normalised} if the following holds: the
first three values $Q_1,Q_2,Q_3$ are $\infty,0,1$ respectively; and
the first three preimages $C_{1,1},C_{2,1},C_{3,1}$ are also
respectively $\infty,0,1$.  This means that we assume that $W_1$ has
degree $d-\alpha_{1,1}$, that $W_{2,1}=x^{\alpha_{2,1}}$, and that
$W_{3,1}=(x-1)^{\alpha_{3,1}}$.

Note that this assumption is not innocuous: it may well be that no
critical point $C_{i,j}$ is defined over $\Bbbk$. The normalization
may be imposed at no cost if (after permutation of the indices)
$\mu_{1,1}=\mu_{2,1}=\mu_{3,1}=1$.

We are now ready to detail the lifting algorithm. Out of coherent data
in $\FF_p$ and a parameter $N$, it computes a $p^{-N}$-approximation
of the corresponding coherent data in $\ZZ_p$, in the form of an
approximation in $\ZZ/p^N$.
\begin{alg}[Lift a solution $p$-adically]\label{aLift}\rule{0mm}{0mm}\\
  \underline{Given:} coherent data $\overline
  W_i=\prod_{j=1}^{n_j}\overline W_{i,j}^{\beta_{i,j}}\in\FF_p[x]$,
  $\overline\lambda\in\FF_p^\times$, and $\overline
  Q_i\in\Pone(\FF_p)$; a parameter $N\in\NN$; and lifts
  $Q_i\in\Pone(\ZZ/p^N)$ of the points $\overline Q_i$\\
  \underline{Return for infinitely many $p$:} coherent data
  $W_i=\prod_{j=1}^{n_j}W_{i,j}^{\beta_{i,j}}\in(\ZZ/p^N)[x]$ and
  $\lambda\in(\ZZ/p^N)^\times$ that reduce mod $p$ to $\overline W_i$.

  First, we assume that the data may be normalised. This amounts to
  requiring at least three of the $W_{i,j}$, for distinct $i$'s, to
  have a linear factor. This holds for a positive proportion of primes
  $p$. If no such three factors exist, the algorithm
  aborts. Otherwise, we silently replace the three corresponding
  $\beta_{i,j}^{\mu_{i,j}}$ by
  $(\beta_{i,j}^1,\beta_{i,j}^{\mu_{i,j}-1})$ in the shapes so as to
  create a term with $\mu_{i,j}=1$.

  We write now each $W_{i,j}$ in the form
  \begin{equation}\label{eq:W}
    W_{i,j}=x^{\mu_{i,j}}+\sum_{s=1}^{\mu_{i,j}}w_{i,j,s}x^{\mu_{i,j}-s},
  \end{equation}
  for unknowns $w_{i,j,s}\in\Z/p^N$.

  Recall from~\eqref{eq:buff} the expressions $F_i=W_i+\lambda
  q_i W_1-W_2$ and $\mathcal F=(F_3,\dots,F_k)$. The $F_i$ are
  polynomials in the variables $\{w_{i,j,s}\colon
  (i,j)\neq(1,1),(2,1),(3,1)\}\cup\{\lambda\}$. We lift the
  co\"efficients of the coherent data $(\overline
  W_i,\overline\lambda)$ to $\Z$, to obtain an initial parameter
  $a^0=(w_{1,2,1}^0,\dots,w_{k,n_k,\mu_{k,n_k}}^0,\lambda^0)$. Since
  the original data is coherent, we have $\mathcal F(a^0)\equiv0\pmod
  p$. For almost all $p$, the Jacobian $D\mathcal F$ is invertible by
  Corollary~\ref{cor:buff}; if $D\mathcal F$ is not invertible at
  $a^0$, then we abort the algorithm. Otherwise, we apply repeatedly
  Hensel's Lemma~\ref{pHensel} to obtain a solution $a$ to $\mathcal
  F(a)\equiv0\pmod{p^N}$.

  Finally, we reconstruct the polynomials $W_{i,j}$ out of their
  co\"efficients (which are just co\"ordinates of $a$).
\end{alg}
\begin{proof}[Proof of validity]
  The invertibility of the Jacobian was expressed in
  Corollary~\ref{cor:buff} in terms of the variables $C_{i,j}$. This
  does not make any difference: here we express them in terms of the
  $w_{i,j,s}$, which are elementary symmetric functions of the
  $C_{i,j}$.
\end{proof}

\begin{example}
  Consider the shapes $\alpha_1=\alpha_2=\alpha_3 =(4^1,3^1,2^3)$. Our
  example
  \begin{align*}
    W_1 &= (x-5)^3(x^3+3x^2+2x+3)^2\\
    W_2 &= x^4 (x+3)^3 (x^3-3x-5)^2\\
    W_3 = W_2+4W_1 &= (x-1)^4(x-3)^3 (x^3-2x-3)^2
  \end{align*}
  gives a vector of co\"efficients
  \begin{align*}
    a^0 &= \begin{array}[t]{c@{}c@{}c}( & w_{1,2,1},w_{1,3,1},w_{1,3,2},w_{1,3,3}, \\
      &w_{2,2,1},w_{2,3,1},w_{2,3,2},w_{2,3,3}, \\
      &w_{3,2,1},w_{3,3,1},w_{3,3,2},w_{3,3,3},&\lambda)
    \end{array}\\
    &= (-5,3,2,3,\; 3,-3,0,-5,\; -3,-2,0,-3,\; -4)
  \end{align*}
  with $F(w_0) \equiv 0\pmod{11}$. The lift is
  \[a^1 = (50,-41,13,25,\;-19,-33,19,-60,\;-47,11,-46,-58,\;51)
  \]
  with $F(w_1) \equiv 0 \pmod{11^2}$. We can continue this process
  inductively. Notice that the precision doubles in every step.
\end{example}	

\section{Promoting a solution from $\ZZ_p$ to a number field $\KK$}
If the Hurwitz problem has a solution over $\ZZ$ that reduces to a
given solution over $\FF_p$, then Hensel lifting will find it after a
finite number of steps. Unfortunately the solutions usually involve
fractional co\"efficients, and are usually defined over a finite
extension $\KK$ of $\QQ$. Our first goal will therefore be to
determine this extension.

Consider a degree-$e$ extension $\KK$ of the rationals, and
$a\in\KK$. Then $1,a,\dots,a^e$ are linearly dependent over $\QQ$, and
therefore also over $\ZZ$, i.e there exists a polynomial
\[P = p_0 + p_1t + \dots + p_e t^e
\]
with all $p_i\in\ZZ$ and $P(a) = 0$. Let now $p\in\NN$ be a prime such
that $P$ splits over $\ZZ_p$; so that we may view $\QQ(a)$ as a
subfield of $\QQ_p$. Assume also that $a$ is invertible modulo $p$,
so that we may consider $a$ as an element of $\ZZ_p$.

Consider now $\tilde a \equiv a \pmod{p^N}$; then we have the equation
\[P(\tilde a) = p_{e+1}  p^N 
\]
which is linear in $\{p_0,\dots,p_{e+1}\}$. We use the LLL
algorithm~\cite{lenstra-l-l:factoring} to find small integer solutions
to this linear equation. The default implementation uses a simple
heuristic to guess the correct precision $N$ and the correct extension
degree: for a initial precision we start with extension degree $e=1$
and increase $e$ until a solution is found (i.e.\ $F(\tilde a) = 0
\mod p^{2N}$) or the computed shortest lattice basis vector norm is
the same for $e$ and $e-1$. If the computed vector norm did not
change, we increase the $p$-adic precision.  If we have a-priori
knowledge about the minimum or maximum expected extension degree, then
it can be passed to the algorithm, which is more likely to find
quickly a solution.

The following algorithm is described as a process that, receiving as
input an infinite feed of ever-more-precise approximations of a
$p$-adic number that is known to be algebraic, produces an infinite
stream of ever-more-likely minimal polynomials of that $p$-adic
number.
\begin{alg}[Convert a $p$-adic number to an algebraic number]\label{aLLL}\rule{0mm}{0mm}\\
  \underline{Given:} approximations, to arbitrary precision, of an
  algebraic number $a\in\ZZ_p$\\
  \underline{Return:} polynomials $P(t)\in\ZZ[t]$ whose likelihood
  converges to $1$ of being the minimal polynomial of $a$, as the
  precision of $a$ improves.

  Assume that, for each $N\in\N$, the algorithm may receive an
  approximation $a_N$, to $N$ base-$p$ digits, of $a$. The element
  $a_N$ is represented as an integer in $\{0,\dots,p^N-1\}$.

  Start with $d=1$ and $N=1$. Then, repeat the following.  Consider
  the lattice in $\R^{d+1}$ generated by the columns of the matrix
  \[M=\begin{pmatrix}
    p^N & -a_N & -a_N^2 & \dots & -a_N^d\\
    0   & 1     & 0   & \dots & 0\\
    0   & 0     & 1   & \dots & 0\\
    \vdots & \vdots & & \ddots & \vdots\\
    0   & 0     & 0   & \dots & 1
  \end{pmatrix}.
  \]
  Using the LLL algorithm, find a vector $(P_0,P_1,\dots,P_d)$ in
  the lattice, of small norm $\theta_{N,d}$. Form the polynomial
  $P(t)=\sum_{i=0}^d P_it^i$.

  If $P(a_N)\equiv0\pmod{p^{2N}}$, or if $d>1$ and
  $\theta_{N,d}=\theta_{N,d-1}$, then output $P$ as a candidate
  polynomial. Repeat then, after having incremented $d$ if the first
  case holds, and doubled $N$ otherwise.
\end{alg}
\begin{proof}[Proof of validity]
  The algorithm repeatedly increases $N$ and $d$. Note that the
  polynomials returned may have degree $<d$, so increasing $d$ is
  harmless, and the precision is increased as soon as increase in
  maximal degree does not improve the solution.

  Let $(P_0,\dots,P_d)$ be a short lattice vector. Then this vector is
  \[M\cdot\;^t(P(a_N)/p^N,P_1,\dots,P_d)\] and in particular
  $P(a_N)\equiv0\pmod{p^N}$. On the other hand, the co\"efficients
  $P_i$ are small, so $P$ is likely to be the minimal polynomial of
  $a$.
\end{proof}

\begin{example}
  Considering our example $\alpha_1 = \alpha_2 = \alpha_3 =
  (4^1,3^1,2^3)$ we found
  \[w_{1,2,1} = {\scriptstyle1400834756308742009361916361765119584358776523123371526525883115012}\in\Z/11^{2^6}
  \]
  and $P_{1,2,1}(w_{1,2,1}) \equiv 0 \pmod{11^{2^6}}$ for
  \[P_{1,2,1}(t) = 39t^6+117t^5+195t^4+195t^3+141t^2+63t+16.\] Higher
  precision values of $w_{1,2,1}$ are also zeroes of the same
  polynomial $P_{1,2,1}$.  We take this as a hint that $P_{1,2,1}$ is
  indeed the minimal polynomial of the co\"ordinate $w_{1,2,1}$ in the
  lift of our finite field solution.
\end{example}

Having found the minimal polynomials $P_{i,j,s}\in\ZZ[t]$ for all
co\"ordinates $w_{i,j,s}$ of our solution vector, we determine the
field $\KK$ on which they are all defined, as the compositum of all
field extensions defined by the $P_{i,j,s}$. If these field extensions
were independent, then we should just consider all zeroes in $\CC$ of
the $P_{i,j,s}$ and return the corresponding rational
functions.

However, in general, the field extensions will be highly dependent. To
simplify notation, let us assume that all $P_{i,j,s}$ are of degree
$e$, and that $\KK$ itself is a degree-$e$ extension. Then there are
$e$ possible values for each co\"ordinate. At this stage we do not
know how to combine these single co\"ordinate solutions to a solution
vector (there are $e^{d(k-2)}$ possible combinations). To solve this
problem we use the following method.

To illustrate our method, consider $a,b \in \KK$ and let $P_a,
P_b,P_{a+b}\in \ZZ[x]$ be minimal polynomials of $a,b,a+b$
respectively. Assume furthermore that $P_a$, $P_b$ and $P_{a+b}$ are
of degree $e$, and let $a_1,\dots a_e$, $b_1,\dots,b_e$ and
$c_1,\dots,c_e$ respectively be approximations over $\CC$ of the
zeroes of $P_a$, $P_b$ and $P_{a+b}$. Consider the $e\times e$ ``root
compatibility matrix'' $M = (M_{i,j})$ defined by
\[M_{i,j} = \begin{cases}
  1 & \text{ if there exists $k$ with }a_i + b_j \approx c_k,\\
  0 & \text{ otherwise.}
\end{cases}
\]
If $M$ is a permutation matrix, then it describes which root $b_j$
should be paired with $a_i$, namely it is characterised by
$M_{i,j}=1$. In this manner, all other co\"ordinates are chosen,
dependent on the first choice of a root of $P_{1,2,1}$.

\begin{example}
  Tentative minimal polynomials of $w_{1,3,1}$ and
  $a=w_{1,2,1}+w_{1,3,1}$ are
  {\footnotesize\begin{align*}
    P_{1,3,1}&=28431t^6+255879t^5+982449t^4+2056509t^3+2465721t^2+1597239t+439138\\
    P_a&=28431t^6+341172t^5+1844856t^4+5660928t^3+10384524t^2+10807344t+5068144
  \end{align*}}

  We obtain the following approximate zeroes over $\CC$ using Brent's
  method, implemented in PARI~\cite{pari:user}; we preserve the
  ordering in which the roots were returned.
  \[\begin{array}{ccc}
    w_{1,2,1} & w_{1,3,1} & w_{1,2,1} + w_{1,3,1} \\ \hline
    -.150 + .807i 		& -1.5   -1.02i  & -2.161-1.184i\\
    -.150 - .807i           	& -1.5   +1.02i  & -2.162+1.184i\\
    -.5   + .440i 		& -2.012 - .377i &  -2-1.462i\\
    -.5   - .440i  		& -2.012 + .377i & -2 + 1.462i\\
    -.850 + .807i		& -.988  - .377i & -1.839-1.184i\\
    -.850 - .807i		& -.988  + .377i & -1.839+1.184i
  \end{array}
  \]
  Now consider the compatibility matrix $M$. It is
  \[\begin{pmatrix}
    0& 0& 0&  1 & 0 &  0\\
    0& 0&  1 & 0& 0 & 0 \\
    0&  1 & 0& 0& 0 & 0\\
    1 & 0& 0& 0& 0 & 0 \\
    0& 0& 0& 0& 0 &  1  \\
    0& 0& 0& 0&  1  & 0 
  \end{pmatrix},
  \]
  i.e.\ $M_{i,j}$ is $1$ precisely when $(w_{1,2,1})_i+(w_{1,3,1})_j$
  approximates one of the values in the last column of the table
  above. We note that $M$ is a permutation matrix. Applying the
  permutation $M$ to the list of roots $(w_{1,3,1})_i$ of $P_{1,3,1}$
  leads to a valid co\"ordinate pairing: now ${(M w_{1,3,1})}_i$
  corresponds to the ${( w_{1,2,1})}_i$.
\end{example}

\begin{example}
  There are situations in which the matrix $M$ is not a permutation
  matrix but nevertheless contains useful information. Consider the
  finite algebraic set
  \[S = \left \{ \big(i,\sqrt{2}\big),\big(-i,\sqrt{2}\big),\big(i,-\sqrt{2}\big),\big(-i,-\sqrt{2}\big) \right\}.
  \]
  In this case the minimal polynomials of the co\"ordinates are
  $x^2+1$ and $y^2-2$ and the root compatibility matrix is $M =
  (\begin{smallmatrix}1 & 1 \\ 1 & 1 \end{smallmatrix})$. All $4$
  possible pairings lead to correct solutions.

  There are also situations in which the matrix $M$ differs from the
  permutation matrix giving the correct identification of
  co\"ordinates, even with exact arithmetic. For example, if $\xi$
  denotes a fifth root of unity, and
  \[S=\left\{\big(\xi,\xi^2\big),\big(\xi^2,\xi^4\big),\big(\xi^3,\xi\big),\big(\xi^4,\xi^3\big)\right\}
  \]
  then the co\"ordinates have respectively $x^4+x^3+x^2+x+1$ and
  $y^4+y^3+y^2+y+1$ as their minimal polynomial, and the root
  compatibility matrix is insufficient to recover $S$. Indeed the sum
  of the co\"ordinates $z=x+y$, which has minimal polynomial
  $z^4+2z^3+4z^2+3z+1$, does not distinguish solutions in $S$ from the
  non-solutions
  \[S'=\left\{\big(\xi,\xi^3\big),\big(\xi^2,\xi\big),\big(\xi^3,\xi^4\big),\big(\xi^4,\xi^2\big)\right\}
  \]
  Note that $S$ and $S'$ are distinguished by the equation $x^2-y$
  which holds in $S$ but not in $S'$. A linear form $ax+by$ with
  $0\neq a\neq b\neq0$ would also do.

  There are also situations with more than two variables in which all
  compatibility matrices between two variables lead to possible
  pairings, but their combined information is not enough to identify
  the correct solutions. For example, consider
  \[S = \left \{ \big(\varepsilon_1\sqrt2,\varepsilon_2\sqrt3,\varepsilon_3\sqrt5\big)\colon \varepsilon_i\in\{\pm1\},\varepsilon_1\varepsilon_2\varepsilon_3=1\right\}.
  \]
  All pairings between two co\"ordinates are allowed, but there are
  $4$ solutions in total, and not $8$.
\end{example}

Above we have used the linear forms $x_i + x_j$ to determine the
compatibility. Our algorithm uses random linear forms to avoid these
problems, or at least make them less probable.

We are now ready to explain the algorithm computing the solutions in
number fields that reduce modulo $p^N$ to given approximate solutions
in $\ZZ_p$. That problem is in fact an instance of the following, more
general problem.

The following algorithm is described as a process that, receiving as
input an algebraic system over $\ZZ$ and an infinite feed of
ever-more-precise $p$-adic approximations of a solution, produces a
stream of algebraic solutions (in the form of minimal polynomials over
$\ZZ$ and complex numbers singling out roots of the minimal
polynomials). The stream eventually exhausts all solutions conjugate
to the $p$-adic solution.
\begin{alg}[Solve $0$-dimensional algebraic systems]\label{aMatch}\rule{0mm}{0mm}\\
  \underline{Given:} a polynomial system of equations $\mathcal
  F=(F_1,\dots,F_m)$ in variables $x_1,\dots,x_m$, having a finite
  number of solutions; and approximations, to arbitrary precision, of
  a solution $(\widehat{x_1},\dots,\widehat{x_m})$ in $\ZZ_p$\\
  \underline{Return:} a number field $\KK\subset\CC$, and exact
  solutions $(x_1^i,\dots,x_m^i)$ in $\KK$, for $i=1,\dots,s$; each
  element of $\KK$ is given by its minimal polynomial and an
  approximation in $\CC$ of a particular root.

  We construct the solutions $(x_1^i,\dots,x_m^i)$ iteratively, entry
  by entry, by constructing partial tables
  $\{(x_1^i,\dots,x_k^i)\colon i=1,\dots,s\}$. We start by an empty
  table, with $s=1$ and $k=0$, and let $\epsilon$ be a small number.

  Then, for each $k=1,\dots,m$, we do the following. Using
  Algorithm~\ref{aLLL}, we compute a likely minimal polynomial $P_k$
  of $\widehat{x_k}$, say of degree $e$. We compute, to precision
  better than $\epsilon$, approximate roots $r_1,\dots,r_e\in\CC$ of
  $P_k$. Set
  \[\delta_1=\min\{|r_i-r_j|\colon 1\leq i\neq j\le e\}.\]
  If $\delta_1<\epsilon$, we halve $\epsilon$ and restart all over.

  We next choose randomly a linear form $L = l_1x_1 + \dots + l_k x_k$
  with $l_k\not=0$. Again using Algorithm~\ref{aLLL}, we compute a
  minimal polynomial $P_L$ for
  $L(\widehat{x_1},\dots,\widehat{x_k})$. If the degree of $P_L$ is
  not divisible by $s$, we choose a different linear form $L$ and
  repeat the above. Otherwise, we let $\delta_1$ be the minimal
  distance between roots of $P_L$. If $\delta_1<\epsilon$, we halve
  $\epsilon$ and restart all over.

  We then compute the $s\times e$ matrix $M=(M_{i,j})$ with
  \[M_{i,j}=\begin{cases} 1 & \text{ if }|P_L(L(x_1^i,\dots,x_{k-1}^i,r_j))|<\frac13\min\{\delta_1,\delta_2\}\|L\|_1,\\
    0 & \text{ otherwise}.\end{cases}
  \]
  Let the degree of $P_L$ be $st$. If $M$ contains $t$ ones per row
  and one one per column, then we replace $s$ by $st$ and replace each
  row $(x_1^i,\dots,x_{k-1}^i)$ is the partial table by $t$ rows
  $(x_1^i,\dots,x_{k-1}^i,r_j)$ for all $j$ such that
  $M_{i,j}=1$. Otherwise, we repeat the step with a different linear
  form or, if that failed more than ten times in a row, we simply skip
  the iteration.

  When the iteration finished with $k=m$, we have obtained $s$
  candidate solutions, which we check algebraically by evaluating
  $\mathcal F$ on them. We output all those that are certifiably valid
  solutions, and restart the algorithm with better approximations of
  the $\widehat{x_1},\dots,\widehat{x_m}$.
\end{alg}
\begin{proof}[Proof of validity]
  First, all the solutions returned are valid, since they were checked
  (using exact algebra) by evaluating $\mathcal F$ on them.

  Let now $(x_1,\dots,x_m)$ be a solution that is conjugate to
  $(\widehat{x_1},\dots,\widehat{x_m})$. In particular, the minimal
  polynomials of the $x_i$ and $\widehat{x_i}$ are the same, so they
  will eventually be found by Algorithm~\ref{aLLL}. Similarly, for
  every linear form $L$ with integer co\"efficients,
  $L(x_1,\dots,x_m)$ and $L(\widehat{x_1},\dots,\widehat{x_m})$ also
  have the same minimal polynomial, so it will also be eventually
  found by Algorithm~\ref{aLLL}.
\end{proof}

We apply this algorithm to the same polynomial
equations~\eqref{eq:buff}. The variables $x_i$ are a relabeling of the
$w_{i,j,s}$ from~\eqref{eq:W}.

\section{Computing the monodromy}\label{ss:mono}
In this section, we detail the second part of the algorithm sketched
in~\S\ref{ss:algo}.

We are given an approximation of a degree-$d$ rational map
$f\in\CC(z)$, as well as an approximation of the critical values
$Q\subset\CC$, and the local degrees
$\alpha_{i,1},\dots,\alpha_{i,\ell_i}$ above each critical value
$Q_i$.  We are asked to compute the monodromy of the covering induced
by $f$.

The first step is to compute a triangulation $\mathscr Q$ of
$\Pone(\CC)$ by arcs of circle, and containing $Q$ among its
vertices. A particularly efficient triangulation is the \emph{Delaunay
  triangulation}. This is a decomposition of $\Pone(\CC)$ into
triangles, such that, for any two triangles with a common edge, the
sum of their opposite angles is $>\pi$. Such a triangulation always
exists; is essentially unique; and may be computed e.g.\
using~\cite{renka:stripack}.

For performance reasons, we refine the triangulation by adding
vertices to it: whenever we encounter a triangle whose ratio
``circumradius / shortest side'' is larger than $1000$, we add its
circumcenter to the triangulation. This process converges, and gives a
reasonably good triangulation in that its triangles are not too acute;
see~\cite{shewchuk:delaunay}.

The dual decomposition $\mathscr Q^\perp$ of the sphere is the
associated \emph{Vorono\"i diagram}. It has one vertex, called a
\emph{dual vertex}, per Delaunay triangle and one edge, called
\emph{dual edge}, across every Delaunay edge.  Each of its edges
$\varepsilon$ is parametrised as the preimage, under a M\"obius
transformation $\mu_\varepsilon$, of the arc $[0,1]$. We denote by $W$
the vertex set of $\mathscr Q^\perp$, and choose a basepoint $*\in
W$. For each $w\in W$, we number arbitrarily the elements of the fibre
$f^{-1}(w)$ as $\{w_1,\dots,w_d\}$. Because $W$ is far from $Q$, there
are $d$ preimages of each $w\in W$, and their computation is
numerically stable.

There are now two strategies, which have both been tested and
implemented. The first one is a bit simpler, but the second one
performs better in practice. Both associate a permutation
$\varsigma_\varepsilon$ of $\{1,\dots,d\}$ with each edge
$\varepsilon\in\mathscr Q^\perp$ from $w'\in W$ to $w''\in W$, in such
a way that the the $f$-lift of $\varepsilon$ that starts at $w'_j$
ends at $w''_{\varsigma_\varepsilon(j)}$. Both are explained below;
assuming them, we finish the description of the algorithm.

Let $(\hat\gamma_i)_{i=1,\dots,k}$ be non-crossing (but possibly
overlapping) paths in $\mathscr Q^\perp$ that start and end in $*$,
cyclically ordered around $*$, such that $\hat\gamma_i$ surrounds once
counterclockwise the point $Q_i$ and no other vertex of $Q$. These
paths may be selected as follows: choose first the path $\hat\gamma_1$
arbitrarily, and mark its edges. Then, for $i=2,\dots,k$, choose the
path $\hat\gamma_i$ in such a manner that it does not cross the
previously chosen paths (i.e., it may follow a marked path, but must
depart from it on the same side as it joined it), and starts at $*$ in
counterclockwise order between the paths $\hat\gamma_{i-1}$ and
$\hat\gamma_1$. These paths $\hat\gamma_i$ are of the following form:
follow some edges; then follow counterclockwise the perimeter of the
cell of $\mathscr Q^\perp$ containing $Q_i$, i.e.\ in counterclockwise
order the perpendiculars of the edges of $\mathscr Q$ touching $Q_i$;
and then follow in reverse the first edges.

These paths form a basis for the fundamental group
$\pi_1(\Pone(\CC)\setminus Q,*)$, compatible with the description
from~\S\ref{ss:hurwitz}. Let
$(\varepsilon_{i,1},\dots,\varepsilon_{i,n_i})$ be the edges along
$\hat\gamma_i$, and compute the permutation
$\sigma_i=\varsigma_{i,1}\cdots\varsigma_{i,n_i}$. Then the monodromy
representation of $f$ is given by the family
$(\sigma_i)_{i=1,\dots,k}$.

\subsection{Connect-the-dots}
For each dual edge $\varepsilon\in\mathscr Q^\perp$, going from $w'$
to $w''$, we do the following. Knowing the spherical distance from
$w'$ to $w''$ and using coarse estimates on $|f'(z)|$, we have an
upper bound on the length of each of the $d$ preimages of
$\varepsilon$. We attempt to match each $w'_j$ with a
$w''_{\varsigma(j)}$ for some permutation $\varsigma\in\sym d$, by
matching each $w'_j$ to the closest $w''_{\varsigma(j)}$. If more than
one match is compatible with the upper bound on the length of an arc
above the arc from $w'$ to $w''$, we subdivide the edge $\varepsilon$.

\begin{alg}[Lifting edges by connect-the-dots]\rule{0mm}{0mm}\\
  \underline{Given:} a rational map $f\in\C(z)$, an edge
  $\varepsilon=\mu_\varepsilon^{-1}[0,1]$ from
  $w'=\mu_\varepsilon^{-1}(0)$ to $w''=\varepsilon^{-1}(1)$ and
  orderings $w'_1,\dots,w'_d$ and $w''_1,\dots,w''_d$ of the preimages
  of $w',w''$ respectively\\
  \underline{Return:} a permutation $\varsigma\in\sym d$ such that the
  $f$-lift of $\varepsilon$ starting at $w'_j$ ends at
  $w''_{\varsigma(j)}$

  More generally, the algorithm computes, for arbitrary $0\le s<t\le
  1$, the matching between $f$-preimages of $\mu_\varepsilon^{-1}(s)$
  and of $\mu_\varepsilon^{-1}(t)$; the solution is provided by the
  matching for $s=0$ and $t=1$. We write $\{w_j(t)\}_{j=1,\dots,d}$
  for the $f$-preimages of $\mu_\varepsilon^{-1}(t)$.

  If there is only one possible match between the sets $\{w_j(s)\}$
  and $\{w_j(t)\}$, given by a permutation $\varsigma$, then the
  algorithm returns that permutation. Otherwise, set $u=(s+t)/2$,
  compute the preimages $\{w_j(u)\}_{j=1,\dots,d}$ of
  $\mu_\varepsilon^{-1}(u)$, recursively compute the matching between
  the $w_j(s)$ and $w_j(u)$ and between the $w_j(u)$ and $w_j(t)$, and
  return the product of the corresponding permutations.
\end{alg}

\subsection{Using two triangulations}
The second algorithm is more efficient, and uses fundamentally the
fact that the arcs in the triangulation $\mathscr Q^\perp$ and in its
$f$-preimage are given by algebraic curves.

We initially compute the Delaunay triangulation $\mathscr C$ on
$f^{-1}(Q)$, and parametrise its edges $e$ via M\"obius
transformations $\nu_e$ such that $e=\nu_e^{-1}([0,1])$.

It is straightforward to lift $\mathscr Q^\perp$ through $f$: its
edges are all the curves defined by equations $(\mu_\varepsilon\circ
f)(z)\in[0,1]$.

\begin{alg}[Lifting edges using two triangulations]\rule{0mm}{0mm}\\
  \underline{Given:} a rational map $f\in\C(z)$, an edge
  $\varepsilon=\mu_\varepsilon^{-1}[0,1]$ from
  $w'=\mu_\varepsilon^{-1}(0)$ to $w''=\varepsilon^{-1}(1)$ and
  orderings $w'_1,\dots,w'_d$ and $w''_1,\dots,w''_d$ of the preimages
  of $w',w''$ respectively\\
  \underline{Return:} a permutation $\varsigma\in\sym d$ such that the
  $f$-lift of $\varepsilon$ starting at $w'_j$ ends at
  $w''_{\varsigma(j)}$

  For each $i=1,\dots,d$, we seek the $j\in\{1,\dots,d\}$ such that
  the lift of $\varepsilon$ starting at $w'_i$ ends at $w''_j$. The
  permutation to return is then the map $i\mapsto j$.

  We first determine in which triangle $T$ of $\mathscr C$ the lift
  $w'_i$ lies. Then we compute whether the lift $\tilde\varepsilon$ of
  $\varepsilon$ starting at $w'_i$ leaves $T$. If this happens, then
  it must cross an edge $e$ of $T$, namely, we have
  $\mu_\varepsilon\circ f\circ \nu_e^{-1}([0,1])\in[0,1]$. This
  entails, firstly, that the imaginary part of $\mu_\varepsilon\circ
  f\circ \mu_e^{-1}$ vanishes, and secondly that its real part belongs
  to $[0,1]$. Both are polynomial conditions imposed on real-valued
  polynomials, and are efficiently computable numerically. We also
  keep track of the point of intersection $\tilde w_e$ of
  $\tilde\varepsilon$ and $e$.

  In that case, we move to the neighbouring triangle $T'$ of $T$ along
  edge $e$, and continue. When we do not detect more intersections
  with edges of $\mathscr C$, we know in which triangle of $\mathscr
  C$ the vertex $w''_j$ lies.

  It may happen that two or more vertices $w''_j$ belong to the same
  triangle $T$ that we have found in the previous paragraph. In that
  case, we let $\tilde w$ denote the last point on $\tilde\varepsilon$
  that was computed --- possibly $w'_i$; it also belongs to $T$. We
  consider in turn all candidates $w''_j$, and compute the straight
  path $\delta_j$ from $\tilde w$ to $w''_j$ and its image
  $f(\delta_j)$. If there exists a unique $j$ such that $f(\delta_j)$
  lies in the two triangles of $\mathscr Q$ to which $\varepsilon$
  belongs, then we have found the desired $j$. If there are no such
  $j$, then we interpolate. Let $t_0\in[0,1]$ be such that $f(\tilde
  w)=\varepsilon(t_0)$. Consider $t\in(t_0,1)$, and those lifts $w\in
  f^{-1}(\varepsilon(t))$ that belong to $T$; we then consider the
  paths $\delta_j$ from $w$ to $w''_j$ as before, and continue with
  increasing $t$.
\end{alg}

Additional care must be taken for tangent crossings of edges (when the
imaginary part of $\mu_\varepsilon\circ f\circ \mu_e^{-1}$ has a
multiple zero), and when vertices of $\mathscr C$ lie on edges of
$\mathscr Q^\perp$ or conversely; these are treated as special
cases. However, because of the necessary crudeness of norm estimates
on $|f'(z)|$ in the first method, this second method is preferable.

\section{An application to dynamical systems}\label{sec:cui}
We recall from the introduction that the \emph{post-critical set} of a
branched self-covering $f:\sphere\to \sphere$ with critical value set
$Q_f$ is
\[P_f:=\bigcup_{n\ge0} f^{\circ n}(Q_f).
\]
We are interested in the case where $P_f$ is finite and we consider
$f$ up to isotopy rel $P_f$; namely, $f\sim g$ if there exists a path
of branched self-coverings from $f$ to $g$ whose post-critical set
moves smoothly.  We say that $f$ is \emph{combinatorially
  equivalent}\footnote{This is sometimes called \emph{Thurston
    equivalence}} to a rational map $F$ if there are
orientation-preserving homeomorphisms $\phi:\sphere\to\Pone(\CC)$ and
$\psi:\sphere \to \Pone(\CC)$ such that $F\circ \phi = \psi\circ f$
and $\phi$ is isotopic to $\psi$ rel $P_f$.

On the one hand, many examples of branched self-coverings can be
constructed combinatorially, via triangulations; for these, it is
natural to consider the maps up to isotopy rel the vertices of the
triangulation. On the other hand, a fundamental theorem by Thurston
points to the rigidity of these objects:
\begin{thm}[Thurston; see~\cite{douady-h:thurston}]\label{thm:thurston}
  Let $f:\sphere\to \sphere$ be branched self-covering with finite
  post-critical set $P_f$. For each $p\in P_f$, set
  $o_p=\operatorname{ppcm}\{\deg_q(f^n)\colon q\in f^{-n}(p),
  n>0\}\in\N\cup\{\infty\}$, and assume that $\sum_{p\in
    P_f}(1-1/o_p)>2$. This condition is usually abbreviated into ``$f$
  has hyperbolic orbispace''.

  Then $f$ is combinatorially equivalent to a rational map if and only
  if $f$ admits no ``Thurston obstruction'', namely, if and only if,
  for every collection $\mathcal C$ of isotopy classes of
  non-peripheral disjoint curves on $\sphere\setminus P_f$, the
  $\QQ\mathcal C$-endomorphism
  \[\mathcal C\ni c\mapsto\sum_{d\in
    f^{-1}(c)\cap\mathcal C}\frac{1}{\deg(f:d\to c)}\cdot d\in \QQ\mathcal C\] 
  has spectral radius $<1$.

  Furthermore, in that case, the rational map is unique up to
  conjugation by a M\"obius transformation.
\end{thm}

\subsection{Cui's Problem}
Cui Guizhen suggested in 2010 that if $f$ is a ``Sierpi\'nski map'',
namely a rational map whose Julia set is a Sierpi\'nski carpet, then
there should exist an essential, non-peripheral, simple curve $\gamma$
such that $f^{-n}(\gamma)$ contains at least two components homotopic
to $\gamma$ rel $P_f$, for some $n$ large enough.  He then found a
counterexample to his suggestion, given combinatorially as follows:

\begin{center}
\begin{tikzpicture}

\begin{scope}
\filldraw[fill=green!25, draw=black] (30:0.7) -- (150:0.7) -- (270:0.7) -- cycle;
\draw circle [radius=3];
\foreach\a/\name in {90/\infty,210/0,330/1} {
  \coordinate (A) at (\a:1.2);
  \coordinate (B) at (\a+60:0.7);
  \filldraw[fill=green!25, draw=black] (\a-60:0.7) -- (A) -- (\a-7:1.6) -- cycle;
  \filldraw[fill=green!25, draw=black] (B) -- (A) -- (\a+7:1.6) -- cycle;
  \filldraw[fill=green!25, draw=black] (\a-7:1.6) -- (\a+7:1.6) -- (\a:2.9) -- cycle;
  \filldraw[fill=green!25, draw=black] (\a:2.9) .. controls (\a+10:2.2) and (\a+50:0.9) .. (B) .. controls (\a+70:0.9) and (\a+110:2.2) .. (\a+120:2.9) -- cycle;
  \draw (\a:0.9) node {$\name$};
}

\draw (0,0) node {\tiny $1$};
\draw (0.8,-0.75) node {\tiny $2$};
\draw (1.12,-0.42) node {\tiny $3$};
\draw (0.3,1) node {\tiny $4$};
\draw (-0.3,1) node {\tiny $5$};
\draw (-1.1,-0.42) node {\tiny $6$};
\draw (-0.8,-0.75) node {\tiny $7$};
\draw (1.75,-1.0) node {\tiny $8$};
\draw (0.0,2) node {\tiny $9$};
\draw (-1.75,-1.0) node {\tiny $10$};
\draw (0,-1.2) node {\tiny $11$};
\draw (0.9,0.8) node {\tiny $12$};
\draw (-0.9,0.8) node {\tiny $13$};
\end{scope}

\begin{scope} [xshift=7cm]
\filldraw[fill=green!25, draw=black] (90:1.5) -- (210:1.5) -- (330:1.5) -- cycle;
\draw circle [radius=2];
\foreach\a/\name in {90/\infty,210/0,330/1} \draw (\a:1.7) node {$\name$};      
\end{scope}

\draw[->,thick,red] (3,1) -- node[above] {$f$ = fold} (5,1);
\draw[->,thick,red] (3,-1) -- node[below] {$i$ = imbed} (5,-1);
\end{tikzpicture}
\end{center}

In that case, $P_f=Q:=\{Q_1,Q_2,Q_3\}$ with $Q_1:=\infty$, $Q_2:=0$
and $Q_3:=1$, and $f$ fixes $P_f$ pointwise. Since $f$ has only three
post-critical points, all curves are peripheral. On the other hand,
the Julia set of $f$ is a Sierpi\'nski carpet, as we now show. For all
$i\in\{1,2,3\}$, let $\mathcal U_i$ be the immediate basin of
$Q_i$. Given $i,j\in\{1,2,3\}$ not necessarily distinct, there exists
up to isotopy a unique properly embedded arc with endpoints at
$Q_i,Q_j$ whose interior avoids $P_f$. Direct inspection of the
triangulation shows that none of these arcs are invariant under $f$ up
to isotopy.  From this it follows that the closures of the $\mathcal
U_i$ are pairwise disjoint, and that the boundary of each $\mathcal
U_i$ is a Jordan domain.  Consider next the preimages of the basins
$\mathcal U_i$. Using the fact that $f$ is hyperbolic, no branching
occurs on their boundaries, so all iterated $f$-preimages of the
$\mathcal U_i$ have disjoint closures and the Julia set is a
Sierpi\'nski carpet as claimed.

Kevin Pilgrim indicated to us a degree-$3$ rational map exhibiting the
same phenomenon (its Julia set is a Sierpi\'nski carpet, and it
contains no self-replicating multicurve): start by the degree-$2$
rational map coming from the Torus endomorphism $z\mapsto(1+i)z$ on
$\C/\Z+i\Z$ via the Weierstrass map $\wp$. Then blow up the edge
between $\wp(0)$ and the fixed point.

From the above picture, it is easy to compute the monodromy action about $Q$: 
the permutations are those given in~\eqref{eq:perm}, namely
\begin{alignat*}{2}
  &\sigma_1&=(1,7,11,2)(3,8)\underline{(4,5)}(6,10)(9,12,13),\\
  &\sigma_2 &=(1,3,12,4)(5,9)\underline{(6,7)}(10,13,11)(2,8),\\
  &\sigma_3 &=(1,5,13,6)(7,10)\underline{(2,3)}(8,11,12)(4,9).
\end{alignat*}
Furthermore, the \underline{underlined} cycles mark which preimage of
a critical value should be fixed. This extra dynamical data is
required to determine the combinatorial equivalence class of $f$, and
it is also sufficient since $\#P_f=3$ so the pure mapping class group
of $(\Pone(\CC),\#P_f)$ is trivial. The orbispace of $f$ is hyperbolic,
because $o_p=\infty$ for each $p\in P_f$.  Because $\#P_f=3$, all
curves on $\Pone(\CC)\setminus P_f$ are peripheral, so no Thurston
obstruction may occur. By Theorem~\ref{thm:thurston}, there is then,
up to M\"obius conjugacy, a unique map with monodromy
$(\sigma_1,\sigma_2,\sigma_3)$, fixing $\infty,0,1$ with local degree
$2$, and such that $\infty,0,1$ are critical points marked by the
cycles $(4,5),(6,7),(2,3)$ respectively.

Therefore, the map computed by our algorithm, after precomposition
with a suitable M\"obius transformation that puts the preimages of $Q$
at the points determined by the cycles $(4,5),(6,7),(2,3)$
respectively, is the required solution.

Our algorithm searched in fact for a map $F$ with $\infty,0,1$ of
order $4$. This is an improvement to searching immediately for the
correct map, because there are three points of order $2$ above each of
$\infty,0,1$, and they may lie in a strict field extension.

It remains to determine the appropriate M\"obius transformation with
which to precompose $F$.  The first author developed an algorithm that
determines, from a rational map given by its co\"efficients, both the
monodromy about the critical values and the identification of critical
points with cycles of the monodromy permutations. This algorithm is
part of the software package \textsc{Img} within the computer algebra
system \textsc{Gap}~\cite{gap4:manual}, and will be described
elsewhere. Note, however, that there are finitely many possibilities
to consider for the sought M\"obius transformation, and the correct
one can be found by inspection.  To find the appropriate one and thus
determine the solution to Cui's problem, it suffices to draw the
preimage of the upper hemisphere under $F$, and to identify on the
picture the appropriate preimages of $\infty,0,1$. In the image below,
$F$ was normalized so that the order-$4$ critical points above
$\infty,0,1$ are at cube roots of unity $1,\omega,\omega^2$
respectively. The appropriate preimages of $\infty,0,1$ are marked by
a small red circle, based on the figure above:

\begin{center}
  \begin{tikzpicture}
    \draw (0,0) node {\includegraphics[width=10cm]{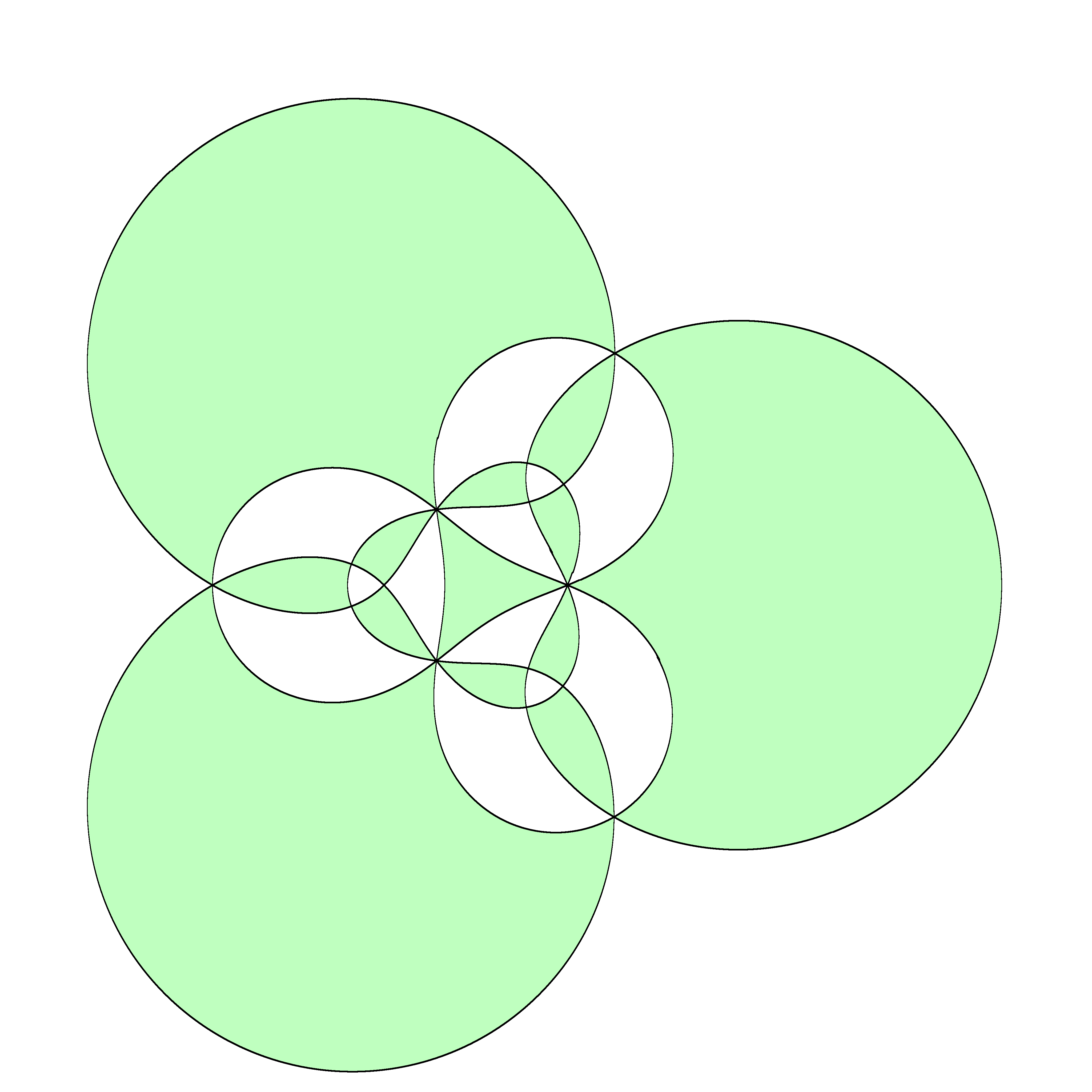}};
    \filldraw[color=red] (-0.16,0.41) circle [radius=2pt];
    \filldraw[color=red] (-0.16,-1.12) circle [radius=2pt];
    \filldraw[color=red] (-1.48,-0.36) circle [radius=2pt];
  \end{tikzpicture}
\end{center}

Our algorithm found a solution $(\bmod\;11)$ of the defining equations
for a map; then lifted them $(\bmod\;11^{2^6})$ and finally obtained
six Galois conjugate solutions. The correct one (with correct choice
of point of order $2$ above $\infty,0,1$) was then found. The original
map is of the form
\[f(z/w)=\lambda\frac{(z-b_4w)^4(z-b_3w)^3(z-b_{2,1}w)^2(z-b_{2,2}w)^2(z-b_{2,3}w)^2}{(z-a_4w)^4(z-a_3w)^3(z-a_{2,1}w)^2(z-a_{2,2}w)^2(z-a_{2,3}w)^2};\]
here are the preimages $a_i,b_i,c_i$ of $\infty,0,1$ respectively:
\begin{align*}
  a_4&=\infty\text{ (meaning the term $z-a_4w$ should be replaced by $1$)},\\
  a_3 &\approx 0.500000000000000000000000000000-0.439846359796987134487167714627i,\\
  a_{2,1} &\approx 1.61268567872451072013417667720-0.490182463946729812334860743821i,\\
  a_{2,2} &\approx 0.500000000000000000000000000000-0.0415300696430258467988035191529i,\\
  a_{2,3} &\approx -0.612685678724510720134176677204-0.490182463946729812334860743821i,\\
  b_4&=0,\\
  b_3 &\approx 1.12748515145901194873474709466-0.991840479188802206853242764751i,\\
  b_{2,1} &\approx 1.98629656633071582984701575517-0.164982069462835473582606346591i,\\
  b_{2,2} &\approx 0.567640411622375679553529964298-0.172536644477962176299255320022i,\\
  b_{2,3} &\approx -0.995164705141609432502666361446-0.796186860797306011242450678339i,\\
  c_4&=1,\\
  c_3 &\approx -0.127485151459011948734747094655-0.991840479188802206853242764751i,\\
  c_{2,1} &\approx 0.432359588377624320446470035702-0.172536644477962176299255320022i,\\
  c_{2,2} &\approx -0.986296566330715829847015755165-0.164982069462835473582606346591i,\\
  c_{2,3} &\approx 1.99516470514160943250266636145-0.796186860797306011242450678339i,\\
  \lambda &\approx 0.130027094895701439414281708196i.
\end{align*}
The required M\"obius transformation $\mu$ maps $(\infty,0,1)$ to $(a_{2,2},b_{2,3},c_{2,2})$
respectively, so
\[\mu(z)=\frac{a_{2,2}(c_{2,2}-b_{2,3})z+b_{2,3}(a_{2,2}-c_{2,2})}{(c_{2,2}-b_{2,3})z+(a_{2,2}-c_{2,2})}
\]

and the required solution is $f\circ\mu$.  Its Julia set is displayed
in Figure~\ref{fig:julia}.

\begin{figure}[ht]
  \centerline{\scalebox{0.3}{\includegraphics[viewport=0 180 750 842]{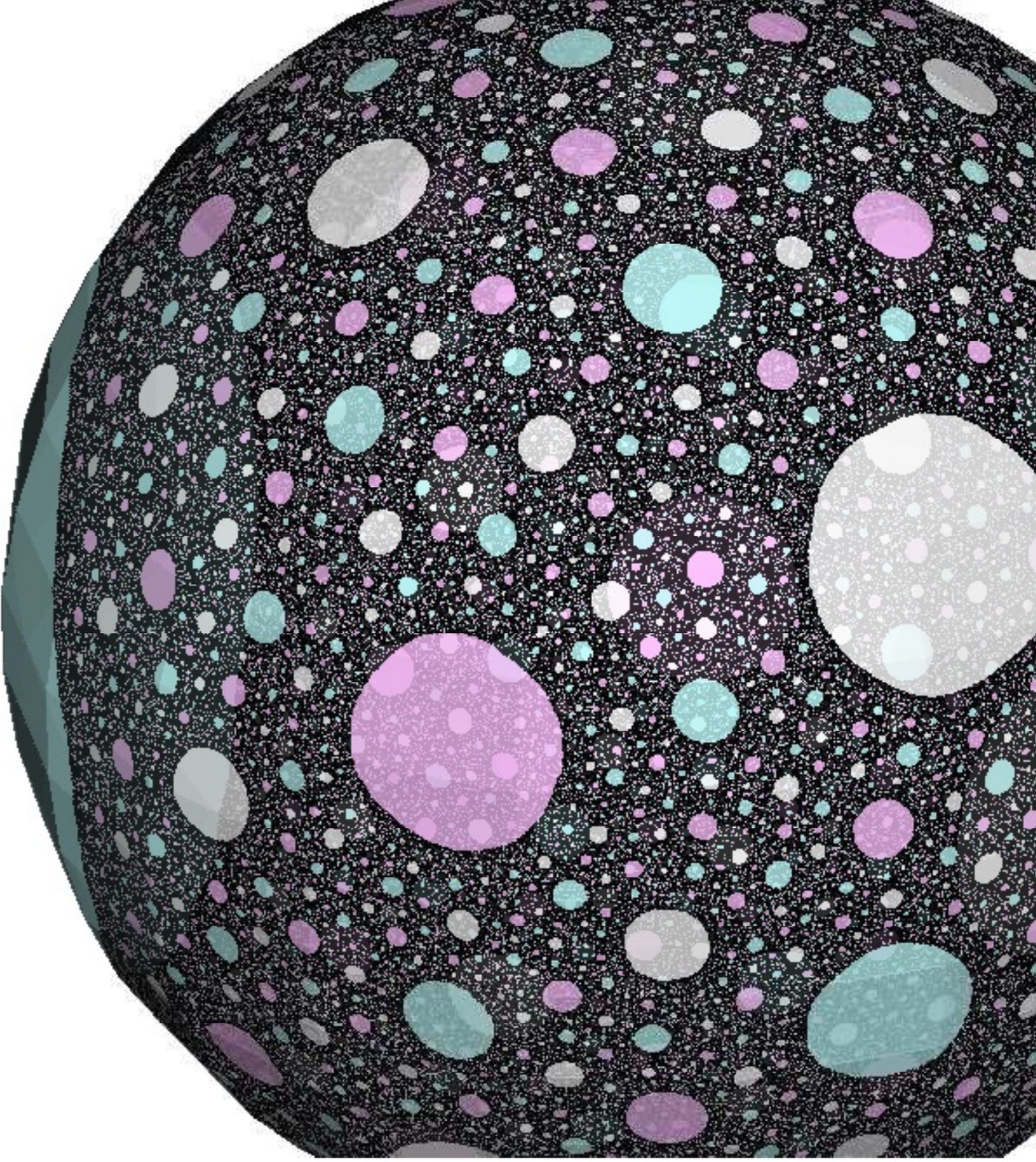}}}
  \caption{Julia set of Cui's map}
  \label{fig:julia}
\end{figure}

\subsection*{Acknowledgments}
We are grateful to Kevin Pilgrim for valuable remarks on a preliminary
version of the text, and for more examples of Sierpi\'nski maps with
three post-critical points. The referee pointed out with great acuity
some deficiencies in the exposition, and contributed a shorter proof
of the assertion that the Julia set of Cui's map is a Sierpi\'nski
carpet.


\bib{atkin-sd:noncongruence}{article}{
  author={Atkin, A. Oliver L.},
  author={Swinnerton-Dyer, H. Peter F.},
  title={Modular forms on noncongruence subgroups},
  conference={ title={Combinatorics (Proc. Sympos. Pure Math., Vol. XIX, Univ. California, Los Angeles, Calif., 1968)}, },
  book={ publisher={Amer. Math. Soc.}, place={Providence, R.I.}, },
  date={1971},
  pages={1--25},
  review={\MR {0337781 (49 \#2550)}},
}

\bib{pari:user}{manual}{
  organization={The PARI~Group},
  title={PARI/GP, version {\tt 2.5.0}},
  year={2011},
  address={Bordeaux},
  url={http://pari.math.u-bordeaux.fr/},
}

\bib{bowers-stephenson:dessins}{article}{
  author={Bowers, Philip L.},
  author={Stephenson, Kenneth},
  title={Uniformizing dessins and Bely\u \i \ maps via circle packing},
  journal={Mem. Amer. Math. Soc.},
  volume={170},
  date={2004},
  number={805},
  pages={xii+97},
  issn={0065-9266},
  review={\MR {2053391 (2005a:30068)}},
}

\bib{boyd-henriksen:medusa}{article}{
  author={Boyd, Suzanne Hruska},
  author={Henriksen, Christian},
  title={The Medusa algorithm for polynomial matings},
  journal={Conform. Geom. Dyn.},
  volume={16},
  date={2012},
  pages={161--183},
  issn={1088-4173},
  review={\MR {2943594}},
  doi={10.1090/S1088-4173-2012-00245-7},
}

\bib{couveignes-granboulan:dessins}{article}{
  author={Couveignes, Jean-Marc},
  author={Granboulan, Louis},
  title={Dessins from a geometric point of view},
  conference={ title={The Grothendieck theory of dessins d'enfants}, address={Luminy}, date={1993}, },
  book={ series={London Math. Soc. Lecture Note Ser.}, volume={200}, publisher={Cambridge Univ. Press}, place={Cambridge}, },
  date={1994},
  pages={79--113},
  review={\MR {1305394 (96b:14015)}},
}

\bib{couveignes:exemples}{article}{
  author={Couveignes, Jean-Marc},
  title={Quelques rev\^etements d\'efinis sur $\mathbb Q$},
  language={French, with French summary},
  journal={Manuscripta Math.},
  volume={92},
  date={1997},
  number={4},
  pages={409--445},
  issn={0025-2611},
  review={\MR {1441485 (98c:14021)}},
  doi={10.1007/BF02678203},
}

\bib{douady-h:thurston}{article}{
  author={Douady, Adrien},
  author={Hubbard, John H.},
  title={A proof of Thurston's topological characterization of rational functions},
  journal={Acta Math.},
  volume={171},
  date={1993},
  number={2},
  pages={263--297},
  issn={0001-5962},
  review={\MR {1251582 (94j:58143)}},
}

\bib{gap4:manual}{manual}{
  title={GAP --- Groups, Algorithms, and Programming, Version 4.4.10},
  label={GAP08},
  author={The GAP~Group},
  date={2008},
  url={\texttt {http://www.gap-system.org}},
}

\bib{grothendieck:esquisse}{article}{
  author={Grothendieck, Alexandre},
  title={Esquisse d'un programme},
  language={French, with French summary},
  note={With an English translation on pp.\ 243--283},
  conference={ title={Geometric Galois actions, 1}, },
  book={ series={London Math. Soc. Lecture Note Ser.}, volume={242}, publisher={Cambridge Univ. Press}, place={Cambridge}, },
  date={1997},
  pages={5--48},
  review={\MR {1483107 (99c:14034)}},
}

\bib{hubbard-s:spider}{article}{
  author={Hubbard, John H.},
  author={Schleicher, Dierk},
  title={The spider algorithm},
  conference={ title={Complex dynamical systems}, address={Cincinnati, OH}, date={1994}, },
  book={ series={Proc. Sympos. Appl. Math.}, volume={49}, publisher={Amer. Math. Soc.}, place={Providence, RI}, },
  date={1994},
  pages={155--180},
  review={\MR {1315537}},
}

\bib{hurwitz:ramifiedsurfaces}{article}{
  author={Hurwitz, Adolf},
  title={Ueber Riemann'sche Fl\"achen mit gegebenen Verzweigungspunkten},
  language={German},
  journal={Math. Ann.},
  volume={39},
  date={1891},
  number={1},
  pages={1--60},
  issn={0025-5831},
  review={\MR {1510692}},
  doi={10.1007/BF01199469},
}

\bib{kuehnau:interpolation}{article}{
  author={K{\"u}hnau, R.},
  title={Numerische Realisierung konformer Abbildungen durch ``Interpolation''},
  language={German, with English and Russian summaries},
  journal={Z. Angew. Math. Mech.},
  volume={63},
  date={1983},
  number={12},
  pages={631--637},
  issn={0044-2267},
  review={\MR {737000 (85b:30012)}},
  doi={10.1002/zamm.19830631206},
}

\bib{lenstra-l-l:factoring}{article}{
  author={Lenstra, Arjen K.},
  author={Lenstra, Hendrik W., Jr.},
  author={Lov{\'a}sz, L\'aszl\'o},
  title={Factoring polynomials with rational coefficients},
  journal={Math. Ann.},
  volume={261},
  date={1982},
  number={4},
  pages={515--534},
  issn={0025-5831},
  review={\MR {682664 (84a:12002)}},
  doi={10.1007/BF01457454},
}

\bib{malle:primitive}{article}{
  author={Malle, Gunter},
  title={Polynomials for primitive nonsolvable permutation groups of degree $d\leq 15$},
  journal={J. Symbolic Comput.},
  volume={4},
  date={1987},
  number={1},
  pages={83--92},
  issn={0747-7171},
  review={\MR {908415 (89b:12007)}},
  doi={10.1016/S0747-7171(87)80056-1},
}

\bib{malle-matzat:realizierung}{article}{
  author={Malle, Gunter},
  author={Matzat, Bernd H.},
  title={Realisierung von Gruppen ${\rm PSL}_2({\bf F}_p)$ als Galoisgruppen \"uber ${\bf Q}$},
  language={German},
  journal={Math. Ann.},
  volume={272},
  date={1985},
  number={4},
  pages={549--565},
  issn={0025-5831},
  review={\MR {807290 (87e:12002)}},
  doi={10.1007/BF01455866},
}

\bib{marshall-rohde:convergence}{article}{
  author={Marshall, Donald E.},
  author={Rohde, Steffen},
  title={Convergence of a variant of the zipper algorithm for conformal mapping},
  journal={SIAM J. Numer. Anal.},
  volume={45},
  date={2007},
  number={6},
  pages={2577--2609 (electronic)},
  issn={0036-1429},
  review={\MR {2361903 (2008m:30008)}},
  doi={10.1137/060659119},
}

\bib{matiyasevich:chebyshev}{article}{
  author={Matiyasevich, Yu. V.},
  title={Computation of generalized Chebyshev polynomials on a computer},
  language={Russian, with Russian summary},
  journal={Vestnik Moskov. Univ. Ser. I Mat. Mekh.},
  date={1996},
  number={6},
  pages={59--61, 112},
  issn={0201-7385},
  translation={ journal={Moscow Univ. Math. Bull.}, volume={51}, date={1996}, number={6}, pages={39--40 (1997)}, issn={0027-1322}, },
  review={\MR {1481512}},
}

\bib{pilgrim:dessins}{article}{
  author={Pilgrim, Kevin M.},
  title={Dessins d'enfants and Hubbard trees},
  language={English, with English and French summaries},
  journal={Ann. Sci. \'Ecole Norm. Sup. (4)},
  volume={33},
  date={2000},
  number={5},
  pages={671--693},
  issn={0012-9593},
  review={\MR {1834499 (2002m:37062)}},
  doi={10.1016/S0012-9593(00)01050-8},
}

\bib{poirier:portraits}{article}{
  author={Poirier, Alfredo},
  title={Critical portraits for postcritically finite polynomials},
  journal={Fund. Math.},
  volume={203},
  date={2009},
  number={2},
  pages={107--163},
  issn={0016-2736},
  review={\MR {2496235 (2010c:37095)}},
  doi={10.4064/fm203-2-2},
}

\bib{renka:stripack}{article}{
  author={Renka, Robert J.},
  title={Algorithm 772: STRIPACK: Delaunay triangulation and Voronoi diagram on the surface of a sphere},
  journal={ACM Trans. Math. Software},
  volume={23},
  date={1997},
  number={3},
  pages={416--434},
  issn={0098-3500},
  review={\MR {1672176}},
  doi={10.1145/275323.275329},
}

\bib{shewchuk:delaunay}{article}{
  author={Shewchuk, Jonathan R.},
  title={Delaunay refinement algorithms for triangular mesh generation},
  note={16th ACM Symposium on Computational Geometry (Hong Kong, 2000)},
  journal={Comput. Geom.},
  volume={22},
  date={2002},
  number={1-3},
  pages={21--74},
  issn={0925-7721},
  review={\MR {1893652 (2003b:65131)}},
  doi={10.1016/S0925-7721(01)00047-5},
}



\begin{bibsection}
\begin{biblist}
\bibselect{math}
\end{biblist}
\end{bibsection}
\end{document}